\documentclass[12pt,psamsfonts,leqno,oneside,letterpaper]{amsart}
\usepackage[dvips,text={6.5truein,9truein},left=1truein,top=1truein]{geometry}
\usepackage{amssymb,amsmath,amscd,enumerate}
\usepackage[pdftex]{graphicx}
\usepackage{tikz}
\usetikzlibrary{shapes,arrows}
\usepackage{url}

\usepackage[colorlinks,linkcolor=blue,citecolor=blue,pdfstartview=FitH]{hyperref}
\input xy
\xyoption{all}
\SelectTips{cm}{12}

\usepackage{color}

\theoremstyle{definition}
\newtheorem{para}{}[section]
\newtheorem{subpara}{}[para]
\newtheorem{remark}[para]{Remark}
\newtheorem{remarks}[para]{Remarks}
\newtheorem{notation}[para]{Notation}
\newtheorem{convention}[para]{Convention}
\newtheorem{definition}[para]{Definition}
\newtheorem{definitions}[para]{Definitions}
\newtheorem{claim}[equation]{Claim}
\newtheorem{example}[para]{Example}
\newtheorem{important}[para]{Important Note}

\newcommand\Alternatives{\begin{enumerate}[(i)]}
\newcommand\EndAlternatives{\end{enumerate}}
\newcommand\Conditions{\begin{enumerate}[(1)]}
\newcommand\EndConditions{\end{enumerate}}

\theoremstyle{plain}
\newtheorem{theorem}[para]{Theorem}
\newtheorem{lemma}[para]{Lemma}
\newtheorem{proposition}[para]{Proposition}
\newtheorem{corollary}[para]{Corollary}
\newtheorem{conjecture}[para]{Conjecture}

\numberwithin{equation}{para}
\numberwithin{figure}{section}

\newcommand\Number{\begin{para}}
\newcommand\EndNumber{\end{para}}
\newcommand\Definition{\begin{definition}}
\newcommand\EndDefinition{\end{definition}}
\newcommand\Definitions{\begin{definitions}}
\newcommand\EndDefinitions{\end{definitions}}
\newcommand\Theorem{\begin{theorem}}
\newcommand\EndTheorem{\end{theorem}}
\newcommand\Conjecture{\begin{conjecture}}
\newcommand\EndConjecture{\end{conjecture}}
\newcommand\Remark{\begin{remark}}
\newcommand\EndRemark{\end{remark}}
\newcommand\Remarks{\begin{remarks}}
\newcommand\EndRemarks{\end{remarks}}
\newcommand\Convention{\begin{convention}}
\newcommand\EndConvention{\end{convention}}
\newcommand\Notation{\begin{notation}}
\newcommand\EndNotation{\end{notation}}
\newcommand\Lemma{\begin{lemma}}
\newcommand\EndLemma{\end{lemma}}
\newcommand\Proposition{\begin{proposition}}
\newcommand\EndProposition{\end{proposition}}
\newcommand\Corollary{\begin{corollary}}
\newcommand\EndCorollary{\end{corollary}}
\newcommand\Claim{\begin{claim}}
\newcommand\EndClaim{\end{claim}}
\newcommand\Proof{\begin{proof}}
\newcommand\EndProof{\end{proof}}
\newcommand\Equation{\begin{equation}}
\newcommand\EndEquation{\end{equation}}

\newcommand\Bullets{\begin{itemize}}
\newcommand\EndBullets{\end{itemize}}

\newcommand\calAC{\mathcal{AC}}
\newcommand\calBC{\mathcal{BC}}
\newcommand\calc{{\mathcal C}}
\newcommand\cale{\mathcal{E}}
\newcommand\calf{\mathcal{F}}

\newcommand\cals{\mathcal{S}}
\newcommand\calt{{\mathcal T}}

\newcommand\co{\colon\thinspace}
\newcommand\bfb{\mathbf{b}}
\newcommand\bc{\mathbf{c}}
\newcommand\bd{\mathbf{d}}

\begin{document}

\title[Bounding areas of centered dual two-cells]{Bounding the area of a centered dual two-cell below, given lower bounds on its side lengths}

\author{Jason DeBlois}
\address{Department of Mathematics\\
University of Pittsburgh\\
301 Thackeray Hall\\
Pittsburgh, PA 15260}
\email{jdeblois@pitt.edu}

\begin{abstract}  For a locally finite set $\cals$ in the hyperbolic plane, suppose $C$ is a compact, $n$-edged two-cell of the centered dual complex of $\cals$, a coarsening of the Delaunay tessellation introduced in the author's prior work.  We describe an effectively computable lower bound for the area of $C$, given an $n$-tuple of positive real numbers bounding its side lengths below, and for $n\leq 9$ implement an algorithm to compute this bound.  For geometrically reasonable side-length bounds, we expect the area bound to be sharp or near-sharp.\end{abstract}

\maketitle

This paper upgrades the centered dual machine, which the author used in \cite{DeB_Voronoi} to give sharp upper bounds on the maximal injectivity radius of complete, orientable, finite-area hyperbolic surfaces.  Two-cells of the centered dual complex are obtained by grouping Delaunay cells that are not ``centered'', with the goal of producing area bounds from side length bounds.  (The construction is more thoroughly reviewed in Section \ref{intro} below.)  Theorem 3.31 of \cite{DeB_Voronoi}, one of the main results of that paper, realizes this goal.  It gives a lower bound in terms of $d>0$, on the area of an arbitrary centered dual two-cell with all edge lengths at least $d$.  

The main result of this paper generalizes and strengthens that one.  Below for a tree $T$ we refer by the \textit{frontier} of $T$ to the collection of edges of some ambient graph that intersect $T$ but do not lie in it; we assume that each frontier edge has exactly one vertex in $T$.

\newcommand\Upgrade{Let $C$ be a compact two-cell of the centered dual complex of a locally finite set $\cals\subset\mathbb{H}^2$ such that for some $\bfb = (b_1,\hdots,b_n)\in(\mathbb{R}^+)^n$ and enumeration of the edges of $C$, the $i^{\mathit{th}}$ edge has length at least $b_i$ for each $i$.  Then $\mathrm{area}(C)\geq\min\{B_T(\sigma(\bfb))\,|\,T\in\calt_n,\sigma\in S_n\}$, where $B_T$ is the area bounding function defined in Proposition \ref{centered lower}, $S_n$ is the symmetric group on $n$ letters, $\sigma\in S_n$ acts on $\bfb$ by permutation of entries, and $\calt_n$ is the collection of compact, rooted trees $T$ with frontier $\calf$ of order $n$ and each vertex trivalent in $V = T\cup\bigcup_{f\in\calf} f$.}
\newtheorem*{UpgradeThrm}{Theorem \ref{upgrade}}
\begin{UpgradeThrm}\Upgrade\end{UpgradeThrm}

For a self-contained but necessarily more elaborate statement of this result see Corollary \ref{better upgrade}.  Theorem \ref{upgrade} generalizes \cite[Thrm.~3.31]{DeB_Voronoi} by allowing different bounds for the lengths of different edges of $C$.  Even when all edge length bounds are the same, the area bound offered by Theorem \ref{upgrade} is stronger than that of its predecessor for $n>4$ (it is identical for $n=3$ or $4$).  See Proposition \ref{bigger than}.  In fact, we expect it to be sharp for ``geometrically reasonable'' edge length bounds, see Remark \ref{reasonable}.

I intend to use this in the future to study arc length spectra of hyperbolic surfaces.  I used Theorem 3.31 of \cite{DeB_Voronoi} there to prove that paper's sharp upper bound on maximal injectivity radius; or, equivalently, on the length of the shortest non-constant geodesic arc based at a point $p$ (by which I mean one with both endpoints at $p$, but possibly a corner there).  Theorem \ref{upgrade} can be used analogously to bound, say, the length of the second-shortest such arc as a function of the length of the shortest.  Problems of this nature arise naturally when studying hyperbolic three-manifolds with totally geodesic boundary; see eg.~\cite{KM}.

The additional strength and generality of Theorem \ref{upgrade} comes at considerable computational expense.  Whereas the bound of \cite[Thrm.~3.31]{DeB_Voronoi} is given by a formula requiring essentially a single computation, evaluating the bound here for a two-cell with $n$ edges requires performing three computational tasks: enumerating $S_n$, enumerating $\calt_n$, and evaluating $B_T(\sigma(\bfb))$ for each one.  The first problem alone has complexity which is at least factorial in $n$.

Section \ref{practice} describes a Python module, \textit{minimizer.py} containing a script \textit{minimize()} that computes the bounds of Theorem \ref{upgrade}.  To do so it calls an existing Python script, \textit{itertools.permutations()}, for enumerating permutations; a script \textit{treecrawler(,)} (also in \textit{minimizer.py}) that computes $B_T(\bfb)$, given $T$ and $\bfb$; and a hand-compiled library, \textit{forest.txt}, of trees in $\calt_n$ for $n\leq 9$.  So it can actually only compute bounds for cells with up to nine edges.

\begin{remark} The ancillary materials include \textit{minimizer.py} and \textit{forest.txt}.  After downloading them only one modification is required to run \textit{minimizer.minimize()} in a Python 2.7.$n$ interpreter (and possibly others).  See the beginning of Section \ref{practice}.\end{remark}

It is certainly possible to write an algorithm to enumerate $\calt_n$ for arbitrary $n$, and hence to remove the limitation to $n\leq 9$.  However enumerating it without redundancy seems more involved, so in any case it is useful to have a classification in low complexity (see Figure \ref{rooted trees}).  And this is enough for our purpose here, which is just to get a sense for how the bounds of Theorem \ref{upgrade} behave.  We carry this out by exploring a few examples in Section \ref{examples}.

Section \ref{intro} introduces the centered dual decomposition and establishes notation.  We prove Theorem \ref{upgrade} in Section \ref{theory} by deepening some aspects of the argument in \cite{DeB_Voronoi}.  Of particular note, Corollary \ref{update} significantly improves Proposition 3.23 of \cite{DeB_Voronoi}, a key result limiting which points in the ``admissible space'' $\overline{\mathit{Ad}}(\bd_{\calf})$ can minimize the area function $D_T$.

\section{The geometric and centered dual decompositions}\label{intro}

Here we will give a brief, conceptual introduction to the subject of this paper, culminating in a description of compact two-cells of the centered dual decomposition determined by a finite subset of a hyperbolic surface.  The picture we describe here is fully fleshed out in \cite{DeB_Voronoi}, and we refer the reader there for details, proofs, and the general case.  In subsection \ref{notation} we establish notation that we will use in the remainder of the paper.

Suppose $\cals$ is a locally finite subset of $\mathbb{H}^2$.  The \textit{Voronoi tessellation} of $\cals$  is a locally finite convex polygonal decomposition of $\mathbb{H}^2$ with two-cells in bijective correspondence with $\cals$.  For each $s\in\cals$, the Voronoi two-cell containing $s$ is
$$V_s = \{x\in\mathbb{H}^2\,|\,d(x,s)\leq d(x,s')\ \mbox{for all}\ s'\in \cals\}.$$
Each Voronoi vertex $v$ is of the form $\bigcap_{i=1}^n V_{s_i}$ for a finite collection $\{s_i\}\subset\cals$ such that $d(s_i,v)\equiv J$ is minimal among all $s\in\cals$.  The \textit{geometric dual two-cell} dual to $v$ is the convex hull of the $s_i$.  It is \textit{cyclic}; ie.~inscribed in a circle, its \textit{circumcircle}, which has radius $J$ and center $v$.  See \cite[\S 5]{DeB_Delaunay}, and Theorem 5.9 there in particular, and cf.~\cite[\S 1]{DeB_Voronoi}.

We say a geometric dual two-cell is \textit{centered} if the center of its circumcircle (ie.~its dual Voronoi vertex) is contained in its interior.  A Voronoi edge $e$ is centered if it intersects its geometric dual edge (which joins $s$ to $s'$ if $e = V_s\cap V_{s'}$) in its interior; if $e$ is not centered we orient it pointing away from its geometric dual edge.  The two notions of centeredness are related: if $C$ is a non-centered geometric dual two-cell then its dual Voronoi vertex is the initial point of a non-centered Voronoi edge, and the geometric dual to the initial vertex of every non-centered Voronoi edge is non-centered \cite[Lemma 2.5]{DeB_Voronoi}.

We use components of the union of non-centered Voronoi edges to organize centered dual two-cells.  A compact such component is a finite, rooted tree $T$ with all edges pointing toward its root vertex $v_T$.  Circumcircle radius increases in the orientation direction of non-centered edges \cite[L.~2.3]{DeB_Voronoi}, so $v_T$ is also characterized as the vertex whose geometric dual two-cell has maximal circumcircle radius.  See Lemma 2.7, Definition 2.8, and Proposition 2.9 of \cite{DeB_Voronoi}.

A compact centered dual two-cell $C$ is either the geometric dual to a single Voronoi vertex contained in only centered edges (the \textit{centered case}), or it is the union of geometric duals to vertices of a compact component of the union of non-centered Voronoi edges.  In this case we say $C$ is dual to $T$ (see \cite[Definition 2.11]{DeB_Voronoi}); in the centered case we say $C$ is dual to $T = \{v\}$, its dual Voronoi vertex.

The \textit{frontier} of a component $T$ of the union of non-centered Voronoi edges is the set of $(e,v)$ such that $e$ is a Voronoi edge not contained in $T$ and $v\in e\cap T$ is a vertex of $e$.  (So if both vertices of $e\not\subset T$ are in $T$ then $e$ contributes twice to the frontier of $T$.)  For such a tree $T$, the \textit{edge set} of the centered dual two-cell $C$ dual to $T$ is the collection of geometric duals to Voronoi edges contributing to the frontier of $T$, counted with multiplicity.  Figure \ref{non-centered five} illustrates all combinatorial possibilities for compact centered dual two-cells with five edges.

\begin{figure}
\includegraphics{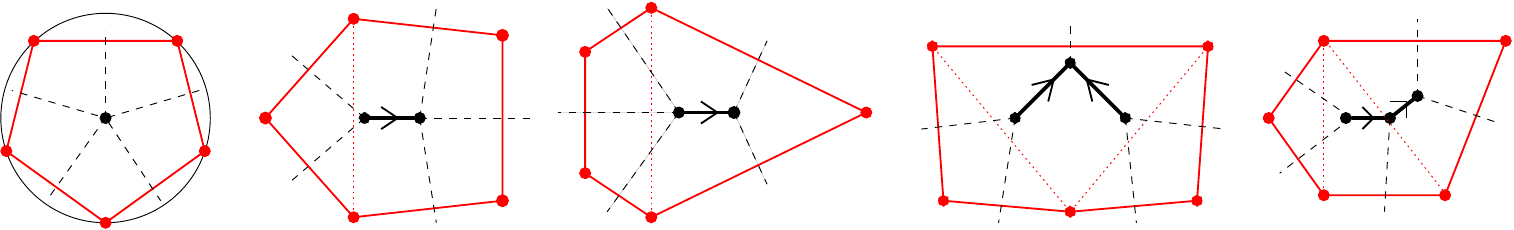}
\caption{The possible five-edged centered dual two-cells, outlined in red, and for each its dual tree $T$ (black, bold) and the frontier of $T$ (black, dashed).}
\label{non-centered five}
\end{figure}

A convex cyclic polygon is determined up to isometry by its set of edge lengths \cite[Prop.~1.8]{DeB_cyclic_geom}, but this is not true of a centered dual two-cell $C$ that is dual to a component $T$ of the union of non-centered Voronoi edges.  However, the geometry of $C$ is \textit{constrained} by its set of edge lengths.  The strategy of \cite{DeB_Voronoi} is to abstract these constraints, defining an \textit{admissible space} parametrizing all possibilities for a two-cell $C$ with fixed combinatorics and edge length collection.  The areas of such possibilities are measured by a continuous function on the admissible space, and we produce area bounds by analyzing minima of this function.

\subsection{Notation}\label{notation} The notation we use for abstracting the study of geometric dual two-cells was introduced in Section 3.2 of \cite{DeB_Voronoi}.  We let $T$ denote a rooted tree and call its root vertex $v_T$.  We always implicitly regard $T$ as embedded in some ambient graph in which each vertex $v$ of $T$ has valence $n_v\geq 3$, and we take the \textit{frontier} $\calf =\{f_1,\hdots,f_n\}$ of $T$ to be the collection of edges of this ambient graph that intersect $T$ but do not lie in it.

We also implicitly assume that each frontier edge $f_i$ has exactly one vertex in $T$.  In the geometric context this could fail; that is, there may exist a component $T$ of the union of non-centered Voronoi edges and a Voronoi edge $e$ with distinct vertices $v$ and $w$, such that $e\cap T = \{v,w\}$. If so then when passing to the abstract context we would denote the edge-vertex pairs $(e,v)$ and $(e,w)$ as $f_i$ and $f_j$, respectively, for some $i\neq j$.

We denote the edge set of $T$ by $\cale$ and study tuples $\bd = (\bd_{\cale},\bd_{\calf})$, where $\bd_{\cale}\in(\mathbb{R}^+)^{\cale}$ and $\bd_{\calf}\in(\mathbb{R}^+)^{\calf}$ are collections of positive real numbers indexed by $\cale$ and $\calf$, respectively.  For a vertex $v$ of $T$ contained in edges $e_1,\hdots,e_{n_v}\in\cale\cup\calf$ and any such $\bd$, we let $P_v(\bd)$ denote the $n_v$-tuple $(d_{e_1},\hdots,d_{e_{n_v}})$ of entries of $\bd$.  

The idea here is that for a centered dual two-cell $C$ dual to $T$, $\bd_{\calf}$ records the set of its edge lengths, since boundary edges of $C$ are dual to frontier edges of $T$ by \cite[Dfn.~2.11]{DeB_Voronoi}.  (Here if both $f_i$ and $f_j$ correspond to a single Voronoi edge as above we take $d_{f_i} = d_{f_j}$ to be the length of the geometric dual to $e$.)  And $\bd_{\cale}$ records the set of lengths of geometric dual edges \textit{internal} to $C$; ie.~edges of intersection between pairs of geometric dual cells contained in $C$.  For such a geometric dual cell with dual Voronoi vertex $v$, $P_v(\bd)$ records its edge length collection, where $\bd = (\bd_{\cale},\bd_{\calf})$. 

Definition 3.10 of \cite{DeB_Voronoi} describes the admissible space $\mathit{Ad}_T(\bd_{\calf})$ of a given $T$ and $\bd_{\calf}$.  For a centered dual two-cell $C$ dual to $T$ with edge length collection $\bd_{\calf}$, with $\bd_{\cale}$ is produced as above, Lemma 3.14 there shows that $\bd = (\bd_{\cale},\bd_{\calf})$ lies in $\mathit{Ad}_T(\bd_{\calf})$.  It is more convenient in practice to deal with a compact space $\overline{\mathit{Ad}}_T(\bd_{\calf})$ containing $\mathit{Ad}(\bd_{\calf})$, which is defined in  \cite[Dfn.~3.15]{DeB_Voronoi}.  We reproduce this below.  There for $v\in T^{(0)}-\{v_T\}$, let $e_v$ be the initial edge of the edge arc joining $v$ to $v_T$.

\begin{definition}[\cite{DeB_Voronoi}, Definition 3.15]\label{admissible closure}For $\bd_{\calf} = (d_e\,|\,e\in\calf)\in(\mathbb{R}^+)^{\calf}$ let $\overline{\mathit{Ad}}(\bd_{\calf})$ consist of those $\bd = (\bd_{\cale},\bd_{\calf})$, for $\bd_{\cale}\in(\mathbb{R}^+)^{\cale}$, such that:\begin{enumerate}
  \item For $v\in T^{(0)} - \{v_T\}$, $P_v(\bd)\in\calAC_{n_v}-\calc_{n_v}$ has largest entry $d_{e_v}$.
  \item $P_{v_T}(\bd)\in\calc_{n_T}\cup\calBC_{n_T}$, where we refer by $n_T$ to the valence $n_{v_T}$ of $v_T$ in $V$.
  \item $J(P_v(\bd))\geq J(P_w(\bd))$ for each $v\in T^{(0)}$ and $w\in v-1$, where $v-1$ is the set of vertices $w\in T^{(0)} - \{v_T,v\}$ such that $v\in e_w$.\end{enumerate}\end{definition}

Here $\calAC_{n}\subset(\mathbb{R}^+)^{n}$ is the open set parametrizing cyclic $n$-gons by their side lengths, see \cite[Corollary 1.10]{DeB_cyclic_geom}.  Its subsets $\calc_n$ and $\calBC_n$  respectively parametrize \textit{centered} and \textit{semicyclic} $n$-gons, those with circumcircle centers in their interiors or, respectively, in a side.  See Propositions 1.11 and 2.2 of \cite{DeB_cyclic_geom}.  Condition (3) above refers to the function $J\co\calAC_n\to\mathbb{R}^+$ that records circumcircle radius of cyclic $n$-gons \cite[Prop.~1.14]{DeB_cyclic_geom}.

\begin{remark}\label{too good}  Below at times we will take $T = \{v_T\}$ (the centered case) as the base case of an inductive argument.  In this case since $\cale$ is empty we omit $\bd_{\cale}$.  Conditions (1) and (3) of Definition \ref{admissible closure} hold vacuously, so appealing to condition (2) we take $\overline{\mathit{Ad}}(\bd_{\calf}) = \{\bd_{\calf}\}$ if $P_{v_T}(\bd_{\calf})\in\calc_{n_T}\cup\calBC_{n_T}$ and $\overline{\mathit{Ad}}(\bd_{\calf}) = \emptyset$ otherwise.\end{remark}

Definition 3.13 of \cite{DeB_Voronoi} introduces the \textit{area function} $D_T(\bd) = \sum_{v\in T^{(0)}} D_0(P_v(\bd))$, where $D_0$ from \cite[Cor.~1.17]{DeB_cyclic_geom} is the function $\calAC_n\to\mathbb{R}$ that measures the area of cyclic $n$-gons \cite[L.~2.1]{DeB_cyclic_geom}.  $D_T$ measures area of centered dual two-cells \cite[L.~3.14]{DeB_Voronoi}.  It is continuous on $\overline{\mathit{Ad}}_T(\bd_{\calf})$ \cite[L.~3.22]{DeB_Voronoi}, and our primary aim here is to understand its minimizers and minima on this set.

\section{Theory}\label{theory}

In this section we will prove Theorem \ref{upgrade}, the generalization of Theorem 3.31 of \cite{DeB_Voronoi} described in the introduction.  We follow the broad strokes of the approach in \cite{DeB_Voronoi}.  First, in Section \ref{minimizer} we prove Corollary \ref{update}, an upgrade of \cite[Prop.~3.23]{DeB_Voronoi} that characterizes minimizers of $D_T$ on $\overline{\mathit{Ad}}_T(\bd_{\calf})$ for a given $\bd_{\calf}$.  We apply this in Section \ref{alla the marbles} to prove Proposition \ref{centered lower}, an improvement on \cite[Prop.~3.30]{DeB_Voronoi} that bounds $D_T(\bd)$ below by a function $B_T(\bfb_{\calf})$ for all $\bd\in\overline{\mathit{Ad}}_T(\bd_{\calf})$ where each entry of $\bd_{\calf}$ is bounded below by the corresponding entry of $\bfb_{\calf}$.

The results above apply to a fixed rooted tree $T$.  We prove Theorem \ref{upgrade} by simply minimizing over all trees and using one new ingredient: Lemma \ref{trivalent}, which allows us to reduce to the trivalent case.  This Lemma reverses the action of Lemma 3.28 of \cite{DeB_Voronoi}, showing for some trees $T$ that $\overline{\mathit{Ad}}_T(\bd_{\calf})\subset \overline{\mathit{Ad}}_{T_0}(\bd_{\calf_0})$ for a related tree $T_0$ with \textit{more} edges than $T$.

\subsection{Minimizers on $\overline{\mathit{Ad}}_T(\bd_{\calf})$}\label{minimizer} Proposition 3.23 of \cite{DeB_Voronoi} supplies a key tool for giving lower bounds on areas of centered dual two-cells, asserting that each minimum (in fact each local minimum) point of $D_T$ on $\overline{\mathit{Ad}}(\bd_{\calf})$ satisfies one of three criteria listed there.  Here we will more closely analyze the situations described there.

\begin{lemma}\label{many polys}For a compact rooted tree $T\subset V$ with root vertex $v_T$, edge set $\cale$, and frontier $\calf = \{f_0,\hdots,f_{k-1}\}$, given $\bd_{\calf} = (d_{f_0},\hdots,d_{f_{k-1}})\in(\mathbb{R}^+)^{\calf}$ and $\bd_{\cale}\in(\mathbb{R}^+)^{\cale}$ let $\bd = (\bd_{\cale},\bd_{\calf})$.  Suppose such a tuple $\bd$ has the following properties:\begin{enumerate}
  \item For each $v\in T^{(0)} - \{v_T\}$, $P_v(\bd)\in \calAC_{n_v}-(\calc_{n_v}\cup\calBC_{n_v})$ has largest entry $d_{e_v}$, where $n_v$ is the valence of $v$ in $V$ and $e_v$ is the initial edge of the arc in $T$ joining $v$ to $v_T$;
  \item $P_{v_T}(\bd)\in\calAC_{n_T}$, where $n_T=n_{v_T}$, and if $P_{v_T}\notin\calc_{n_T}$ then $e\in\calf$ if $d_e$ is maximal among entries of $P_{v_T}(\bd)$; and
  \item $J(P_v(\bd)) = J(P_{v_T}(\bd))$ for each $v\in T^{(0)}$, where $J(P_v(\bd))$ is the circumcircle radius of the cyclic polygon with edge length collection $P_v(\bd)$.\end{enumerate}
Then $\bd_{\calf}$ is in $\calAC_k$, and in $\calc_k$ or $\calBC_k$ if and only if $P_{v_T}(\bd)$ is in $\calc_{n_T}$ or $\calBC_{n_T}$, respectively.  If $\bd_{\calf}\notin\calc_k$ then there is a unique $f_i\in\calf$ such that $d_{f_i}$ is maximal, with $v_T\in f_i$ and $d_{f_i}>d_e$ for all $e\in(\cale\cup\calf)-\{f_i\}$. Also, $J(\bd_{\calf}) = J(P_v(\bd))$ for each $v\in T^{(0)}$, and\begin{align}\label{radius equal sum} D_0(\bd_{\calf}) = \sum_{v\in T^{(0)}} D_0(P_v(\bd)).\end{align}
Moreover, a cyclic $k$-gon with side length collection $\bd_{\calf}$ is tiled by cyclic $n_v$-gons $P_v$ with side length collections $P_v(\bd)$, where $v$ runs over all vertices of $T$.\end{lemma}

\begin{proof}  Note that if $T=\{v_T\}$ then hypotheses (1) and (3) hold vacuously, and the result is a tautology.  Below we will first address the case that $T$ has one edge, then prove the general case by induction.  The one-edge case is an altered version of Lemma 3.25 of \cite{DeB_Voronoi}, with a stronger conclusion and a subtly stronger hypothesis.  That result takes as input an $m$-tuple $\bc_0 = (c_0,\hdots,c_{m-1})$ and an $n$-tuple $\bd_0 = (d_0,\hdots,d_{n-1})$.  We will apply it here in the one-edged case with $\bd_0 = P_{v_T}(\bd)$ and $\bc_0=P_v(\bd)$, where $v$ is the other vertex of $T$.  Our hypothesis (3) above implies the first bulleted hypothesis there, that $J(\bc_0) = J(\bd_0)$; (1) implies the second bullet there with $d_{e_v}$ here in the role of $d_0 = c_0$ there, where $e_v$ is the sole edge of $T$; and (2) here implies the third.  In the notation of \cite[L.~3.25]{DeB_Voronoi}, (2) in fact asserts:
\begin{itemize} \item $\bd_0\in\calAC_n$, and if $\bd_0\notin\calc_n$ then $d_0$ is not maximal among the $d_i$.\end{itemize}
This excludes the possibility that $\bd_0\in\calBC_n$ has maximal entry $d_0$, which was allowed in the third bulleted hypothesis of \cite[L.~3.25]{DeB_Voronoi}.  In any case, since its hypotheses are satisfied the proof and conclusions of that result hold.

It asserts that $\bd\doteq (c_1,\hdots,c_{m-1},d_1,\hdots,d_{n-1})$ is in $\calAC_{m+n-2}$.  Note that since $\cale=\{e_v\}$ has only one element, each of $\bd_0 = P_{v_T}(\bd)$ and $\bc_0 = P_v(\bd)$ has all of its entries but $d_{e_v} = d_0 = c_0$ from $\bd_{\calf}$, so $\bd_{\calf} = \bd$ and $k = m+n-2$.  The conclusion of Lemma 3.25 of \cite{DeB_Voronoi} thus asserts in our terms that $\bd_{\calf}$ is in $\calAC_k$, and in $\calc_k\cup\calBC_k$ if and only if $P_{v_T}(\bd)\in\calc_{n_T}\cup\calBC_{n_T}$; and that $D_0(\bd_{\calf})=D_0(P_v(\bd))+D_0(P_{v_T}(\bd))$.

For the one-edged case of our result we must strengthen this conclusion with four additional assertions (assuming the bulleted hypothesis above).  We claim first that $\bd_{\calf} = \bd$ is the side length collection of a cyclic polygon $P$ tiled by polygons $P_v$ and $P_{v_T}$ with respective side length collections $P_v(\bd) = \bc_0$ and $P_{v_T}(\bd) = \bd_0$.  This is in fact recorded in the proof of \cite[Lemma 3.25]{DeB_Voronoi}, where $P_{v}$ is called $P_0$ and $P_{v_T}$ is $Q_0$.  It implies our second additional assertion, that $J(\bd_{\calf}) = J(P_v(\bd)) = J(P_{v_T}(\bd))$, since $P$ shares a circumcircle with $P_v$ and $P_{v_T}$.

We also need that if $\bd_{\calf}\notin\calc_{k}$ then its unique largest entry comes from among the $d_{f_i}$ with $v_T\in f_i$, i.e.~from among the $d_i\neq d_0$, in the language of \cite[L.~3.25]{DeB_Voronoi}.  The proof there shows that $P_0\cap Q_0 = \gamma_0$ is a side of each with length $c_0 = d_0 = d_{e_v}$, and that $P_0$ and the circumcircle center $v$ lie in opposite half-spaces bounded by the geodesic containing $\gamma_0$.  Proposition 2.2 of \cite{DeB_cyclic_geom} then implies that $\gamma_0$ is unique with this property among sides of $P_0$.  If $\bd_{\calf}\notin\calc_k$, i.e.~$\bd\notin\calc_{m+n-2}$, then the unique longest side $\gamma_{i_0}$ of $P$ is characterized by the fact that $v$ and $P$ lie in opposite half-spaces bounded by the geodesic containing $\gamma_{i_0}$, again by \cite[Prop.~2.2]{DeB_cyclic_geom}.  This implies $\gamma_{i_0}$ is a side of $Q_0$ other than $\gamma_0$, since these comprise the remaining sides of $P$, and this assertion follows.

We finally require that $\bd_{\calf}$ is in $\calc_k$ or $\calBC_{k}$ if and only if $P_{v_T}(\bd) = \bd_0$ is in $\calc_n$ or $\calBC_n$ respectively.  It is asserted in \cite[L.~3.25]{DeB_Voronoi} that $\bd\in\calc_{m+n-2}\cup\calBC_{m+n-2}$ if and only if $\bd_0\in\calc_n\cup\calBC_n$.  We have the additional fact, again by \cite[Prop.~2.2]{DeB_cyclic_geom}, that $\bd\in\calBC_{m+n-2}$ if and only if $v$ lies in a side of $P$.  If this is so then by the above $v$ lies in a side of $Q_0$, so $\bd_0\in\calBC_n$.  On the other hand, our strengthened hypothesis does not allow $v\in\gamma_0$, so if $\bd_0\in\calBC_n$ then $v$ lies in a side of $Q_0$ that is a side of $P$.  Hence $\bd\in\calBC_{m+n-2}$, and the second assertion is proved.  This gives the one-edge case of the current result.

We now proceed to the inductive step.  Let $T$ be a compact rooted tree with at least two edges, and let $v_0$ be a vertex  farthest from $v_T$ in $T$.  Then $v_0$ is contained in a single edge $e_{v_0}$ of $T$, and we take $T_0 = \overline{T-e_{v_0}}$.  Listing the edges containing $v_0$ as $e_{v_0},f_{i_1},\hdots,f_{i_{n_{v_0}-1}}$, where all $f_{i_j}\in\calf$, the frontier of $T_0$ in $V$ is $\calf_0 = (\calf-\{f_{i_j}\})\cup\{e_{v_0}\}$ and the edge set is $\cale-\{e_{v_0}\}$.  Given $\bd = (\bd_{\cale},\bd_{\calf})$ we obtain $\bd_0 = (\bd_{\cale_0},\bd_{\calf_0})$ by omitting the entries $d_{f_j}$ of $\bd_{\calf}$ and shuffling $d_{e_{v_0}}$ from $\bd_{\cale}$ to $\bd_{\calf_0}$.  Then $P_v(\bd_0) = P_v(\bd)$ for all $v\in T_0^{(0)}$.

In particular, if $T$ and $\bd$ satisfy (1)--(3) then so do $T_0$ and $\bd_0$.   We suppose this is so, and assume by induction that the desired conclusion holds for $T_0$ and $\bd_0$.  We now note that the hypotheses of our strengthened \cite[Lemma 3.25]{DeB_Voronoi} (from the one-edged case) are satisfied with $\bc_0$ the $n_{v_0}$-tuple of dual lengths to the edges containing $v_0$ and $\bd_0 = \bd_{\calf_0}$, ordered so that $c_0 = d_0 = d_{e_{v_0}}$.  Applying that result and noting that $\bd$ as described there is $\bd_{\calf}$, we thus conclude that $\bd_{\calf}\in\calAC_n$; it is in $\calc_n$ or $\calBC_n$ if and only if $P_{v_T}(\bd)$ is in $\calc_{n_T}$ or $\calBC_{n_T}$, respectively; that $J(\bd_{\calf} = J(P_v(\bd))$ for all $v\in T^{(0)}$; and that $D_0(\bd_{\calf}) = \sum_{v\in T^{(0)}} D_0(P_v(\bd))$.

Moreover by induction the cyclic $n_0$-gon $Q_0$ of \cite[L.~3.25]{DeB_Voronoi}, with side length collection $\bd_{\calf_0}$, is itself tiled by copies of the $P_v$ for all $v\in T_0^{(0)}$.  And since $P_0$ from \cite[L.~3.25]{DeB_Voronoi} shares the side length collection of $P_{v_0}$ it is isometric to it; hence a cyclic $n$-gon with side length collection $\bd_{\calf}$ is tiled by copies of the $P_v$ as desired.\end{proof}

\begin{lemma}\label{deform many polys}  With the hypotheses of Lemma \ref{many polys}, for any continuous map $t\mapsto \bd_{\calf}(t)\in(\mathbb{R}^+)^{\calf}$ with $\bd_{\calf}(0) = \bd_{\calf}$ there is a continuous map $t\mapsto\bd_{\cale}(t)\in(\mathbb{R}^+)^{\cale}$ on $[0,\epsilon)$ for some $\epsilon>0$, with $\bd_{\cale}(0) = \bd_{\cale}$, such that $\bd(t) = (\bd_{\calf}(t),\bd_{\cale}(t))$ has properties (1) - (3) of the lemma.\end{lemma}

\begin{proof}  We use the assertion in Lemma \ref{many polys} that for $\bd = (\bd_{\calf},\bd_{\cale})$ satisfying its hypotheses, a cyclic $n$-gon $P$ with side length collection $\bd_{\calf}$ is tiled by copies of $P_v$ for $v\in T^{(0)}$.  Thus for each $e\in\cale$ and each vertex $v$ of $e$, the corresponding edge of $P_v(\bd)$ is a diagonal of $P$.  So its length $d_e$ is given by $\ell_{ij}(\bd_{\calf})$ for some fixed $i$ and $j$ between $0$ and $n-1$, where $\ell_{ij}$ is the diagonal-length function of \cite[Corollary 1.15]{DeB_cyclic_geom}.

Now any continuous deformation $\bd_{\calf}(t)$ of $\bd_{\calf}$ remains in $\calAC_n$ for small $t$, since $\calAC_n$ is open in $(\mathbb{R}^+)^n$.  We define $\bd_{\cale}(t)$ by taking $d_e(t) = \ell_{ij}(\bd_{\calf}(t))$ for each $e\in\cale$, where $\ell_{ij}$ is the diagonal-length function described above.  Then by the definition of the $\ell_{i,j}$, $P_v(\bd(t))\in\calAC_{n_v}$ and $J(P_v(\bd(t))) = J(P_{v_T}(\bd(t))) = J(\bd_{\calf}(t))$ for all $t$ such that $\bd_{\calf}(t)\in\calAC_n$.  Criteria (1) and (2) of the lemma involve only open conditions and so hold for $t$ small enough.\end{proof}

\begin{proposition}\label{equal radii}  For a compact rooted tree $T\subset V$ with root vertex $v_T$, edge set $\cale$ and frontier $\calf = \{f_0,\hdots,f_{n-1}\}$, and $\bd_{\calf}=(d_{f_0},\hdots,d_{f_{n-1}})\in(\mathbb{R}^+)^{\calf}$, suppose a local minimum of $\bd_{\cale}\mapsto D_T(\bd_{\cale},\bd_{\calf})$ on $\overline{\mathit{Ad}}(\bd_{\calf})$ occurs at $\bd_{\cale}$ with the following property: for $\bd = (\bd_{\cale},\bd_{\calf})$ there exists $v_0\in T^{(0)}$ such that $J(P_{v}(\bd))=J(P_{v_0}(\bd))$ for the terminal vertex $v$ of $e_{v_0}$.  Then for the maximal subtree $T_0$ of $T$ containing $v_0$ such that $J(P_v(\bd)) = J(P_{v_0}(\bd))$ for all $v\in T_0^{(0)}$:\begin{itemize}
	\item $v_T\in T_0$;
	\item $P_{v_T}(\bd)\in\calBC_{n_T}$, where $v_T$ has valence $n_T$ in $V$; and
	\item for $e_T\in\cale\cup\calf$ containing $v_T$ such that $d_{e_T}$ is maximal among all such edges, either $e_T\in\calf$ or $e_T$ is an edge of $T_0$ and its other endpoint $v_T'$ is on the boundary of $T_0$; ie.~it is of valence one in $T_0$.\end{itemize}
Moreover, if $e_T\in\cale$ then for each $v\in v_T'-1$, $P_v(\bd)\in\calBC_{n_v}$.\end{proposition}

Recall that Definition \ref{admissible closure}(1) asserts for each $v\in T^{(0)}-\{v_T\}$ that $P_v(\bd)$ is in $\calAC_{n_v}-\calc_{n_v}$. We begin by noting a stronger fact for vertices of $T_0$.

\begin{lemma}\label{not in BC}  With hypotheses and notation as in Proposition \ref{equal radii}, let $v_1$ be the nearest vertex of $T_0$ to $v_T$ in $T$ (so $v_T\in T_0 \Leftrightarrow v_T=v_1$). If $v_1\neq v_T$ then for each $v\in T_0^{(0)}-\{v_1\}$, $P_v(\bd)\in\calAC_{n_v}-(\calc_{n_v}\cup\calBC_{n_v})$.  If $v_1=v_T$ then this still holds for all but at most one $v\in T_0^{(0)}-\{v_1\}$.  If in this case there does exist $v\in T_0^{(0)} - \{v_1\}$ with $P_v(\bd)\in\calBC_{n_v}$, then $e_v$ joins $v$ to $v_T$, $P_{v_T}(\bd)\in\calBC_{n_T}$, and both $P_{v_T}(\bd)$ and $P_v(\bd)$ have maximal entry $d_{e_v}$.\end{lemma}

\begin{proof}For $v$ as above we just need to show that $P_v(\bd)\notin\calBC_{n_v}$.  Note that for such $v$, $P_v(\bd)$ has $d_{e_v}$ as its largest entry, by Definition \ref{admissible closure}(1), and by Proposition 2.2 of \cite{DeB_cyclic_geom}, $d_{e_v}$ is \textit{unique} with this property.  Since $v\neq v_1$, the other endpoint $v'$ of $e_v$ lies in $T_0$, so $J(P_v(\bd)) = J(P_{v'}(\bd))$.  If $P_v(\bd)$ lies in $\calBC_{n_v}$ then by \cite[Prop.~1.11]{DeB_cyclic_geom}, $J(P_v(\bd))=d_{e_v}/2$.  Hence also $J(P_{v'}(\bd)) = d_{e_v}/2$, so by the same result $P_{v'}(\bd)\in\calBC_{n_{v'}}$ has largest entry $d_{e_v}$.

This is a contradiction if $v'\neq v_T$, since then it has unique largest entry $d_{e_{v'}} > d_{e_v}$ by Definition \ref{admissible closure}(1).  If $v' = v_T$ then since $v'\in T_0$ we must have $v_1 = v_T$, and the Lemma follows from the previous paragraph.\end{proof}

For the sake of readability we will prove the first two assertions of the Proposition separately.

\begin{lemma}\label{halfway there}With the hypotheses and notation of Proposition \ref{equal radii}, $v_T\in T_0$ and $P_{v_T}\in\calBC_{n_T}$.\end{lemma}

\begin{proof} Suppose that either $v_T$ is not in $T_0$ (ie.~$v_1\neq v_T$), or $v_T\in T_0$ ($v_1=v_T$) and $P_{v_T}(\bd)\in\calc_{n_T}$. So in particular, $P_v(\bd)\notin\calBC_{n_v}$ for all $v\in T_0^{(0)}-\{v_1\}$ by Lemma \ref{not in BC}.  Without loss of generality assume $v_0$ is a farthest vertex of $T_0$ from $v_T$, and refer by $e_0$ to $e_{v_0}$.  Let $T_1 = T_0-(\mathit{int}(e_0)\cup v_0)$, and take $v_1$ as its root vertex.  Let $\calf_1$ be the frontier of $T_1$ in $V$ and name its edge set $\cale_1$, and let $\bd_1 = (\bd_{\cale_1},\bd_{\calf_1})$ take its entries from $\bd$.  Define a deformation $\bd_{\calf_1}(t)$ of $\bd_{\calf_1}$ by taking $d_e(t)\equiv d_e$ for all $e\in\calf_1 - \{e_0\}$ and $d_{e_0}(t) = d_{e_0}-t$.  We will show that this determines a deformation $\bd(t)\in\overline{\mathit{Ad}}(\bd_{\calf})$ such that $D_T(\bd(t))$ is decreasing.

We first note that $\bd_1$ satisfies criteria (1)--(3) of Lemma \ref{many polys}: property (1) is Lemma \ref{not in BC}, and (3) is inherited from $T_0$.  If $v_1\neq v_T$ then criterion (2) follows from the facts that $P_{v_1}(\bd_1) = P_{v_1}(\bd)\in\calAC_{n_{v_1}}-\calc_{n_{v_1}}$ has maximal entry $d_{e_{v_1}}$, and $e_{v_1}\in\calf_1$ since by construction $v_1$ is nearest $v_T$ in $T_1$.  If $v_1 =v_T$ then by hypothesis $P_{v_T}(\bd)\in\calc_{n_T}$, and criterion (2) is immediate.

Lemma \ref{many polys} now implies that if $v_1\neq v_T$ and hence $P_{v_1}(\bd_1)\in\calAC_{n_{v_1}}-\calc_{n_{v_1}}$, then $\bd_{\calf_1}\in\calAC_{n_1}-\calc_{n_1}$, where $n_1 = |\calf_1|$, and it has maximal entry $d_{e_{v_1}}$. Otherwise $\bd_{\calf_1}\in\calc_{n_1}$.  We claim:

\begin{subpara}\label{not in C} If $\bd_{\calf_1}\in\calAC_{n_1}-\calc_{n_1}$ then $\bd_{\calf_1}(t)$ remains in $\calAC_{n_1}-\calc_{n_1}$ for small $t>0$; otherwise $\bd_{\calf_1}(t)\in\calc_{n_1}$ for small $t>0$.\end{subpara}

\begin{proof}[Proof of \ref{not in C}]This holds if $\bd_{\calf_1}\in\calAC_{n_1} - (\calc_{n_1}\cup\calBC_{n_1})$, or if $\bd_{\calf_1}\in\calc_{n_1}$, simply because these sets are open.  If $\bd_{\calf_1}$ lies in $\calBC_{n_1}$ then Proposition 1.12 of \cite{DeB_cyclic_geom} implies that $d_{e_{v_1}} = b_0(d_{e_0},\hdots,d_{e_{n_1-1}})$ for the function $b_0$ defined there, where we have enumerated $\calf_1$ as $\{e_0,\hdots,e_{n_1-1},e_{v_1}\}$.  Since $b_0$ strictly increases in each variable (also by that result), and $d_{e_0}(t) < d_{e_0}$ for each $t>0$, it follows that $b_0(d_{e_0}(t),\hdots,d_{e_{n_1-1}}(t)) < b_0(d_{e_1},\hdots,d_{e_{n_1-1}})$ for each such $t$.  Then $d_{e_{v_1}}(t) \equiv d_{e_{v_1}}$ exceeds $b_0(d_{e_0}(t),\hdots,d_{e_{n_1-1}}(t))$ for each $t>0$, so the claim follows in this case from \cite[Cor.~4.10]{DeB_cyclic_geom}.\end{proof}

By Lemma \ref{deform many polys}, $\bd_{\calf_1}(t)$ determines a deformation $\bd_{\cale_1}(t)$ of $\bd_{\cale_1}$ for small $t\geq 0$ such that properties (1) - (3) of Lemma \ref{many polys} continue to hold for $\bd_1(t) = (\bd_{\cale_1}(t),\bd_{\calf_1}(t))$.  We extend $\bd_1(t)$ to $\bd(t) = (\bd_{\cale}(t),\bd_{\calf}(t))$ by taking $d_e(t)\equiv d_e$ for each $e\in\cale\cup\calf - (\cale_1\cup\calf_1)$.  In particular $\bd_{\calf}(t) \equiv \bd_{\calf}$ since $e_0\in\cale$.  We claim that $\bd(t)\in\overline{\mathit{Ad}}(\bd_{\calf})$ for small enough $t>0$.

The only edges of $T$ that change length under $\bd(t)$ are edges of $T_0$, so if $P_v(\bd)$ changes under $\bd(t)$ then $v\in T_0$.  Property (1) of Definition \ref{admissible closure} is thus immediate for $v\in T^{(0)}-T_0^{(0)}$.  It follows for $v\in T_0^{(0)}-\{v_1\}$ from Lemma \ref{not in BC} and the fact that $\calAC_{n_v} - (\calc_{n_v}\cup\calBC_{n_v})$ is open in $(\mathbb{R}^+)^{n_v}$, and, in the case $v_1\neq v_T$, for $v_1$ by combining \ref{not in C} above with the first assertion of Lemma \ref{many polys}.  Property (2) of Definition \ref{admissible closure}, that $P_{v_T}(\bd(t))\in\calc_{n_T}\cup\calBC_{n_T}$, is immediate if $v_1\neq v_T$ and otherwise follows from \ref{not in C} and Lemma \ref{deform many polys}.

For property (3) of Definition \ref{admissible closure} we must separately consider several possibilities for $v\in T^{(0)}$ and $w\in v-1$. If neither $v$ nor $w$ lies in $T_0$ then we have
$$J(P_v(\bd(t)))\equiv J(P_v(\bd)) \geq J(P_w(\bd)) \equiv J(P_w(\bd(t))) $$
for all $t$.  If $v$ is not in $T_0$ but $w$ is then $w = v_1\neq v_T$, and by definition of $T_0$ the initial inequality is strict: $J(P_v(\bd))>J(P_{v_1}(\bd))$.  So it is preserved for small $t>0$.  The same idea holds if $v\in T_0^{(0)}$ but $w\notin T_0$: the strict initial inequality $J(P_v(\bd))>J(P_w(\bd))$ is preserved for small $t>0$.  If $v$ and $w$ lie in $T_1$---ie.~$v,w\in T_0$ and $w\neq v_0$---then by Lemmas \ref{deform many polys} and \ref{many polys} $J(P_v(\bd(t)))\equiv J(\bd_{\calf_1}(t)) \equiv J(P_w(\bd(t)))$.  So the claim is finally proved by establishing property (3) in the case $w=v_0$, so $v\in T_1$.  This follows from:

\begin{subpara}\label{smaller J} For $t>0$, $J(\bd_{\calf_1}(t)) > J(P_{v_0}(\bd(t)))$.\end{subpara}

\begin{proof}[Proof of \ref{smaller J}]  We have $\frac{d}{dt} J(\bd_{\calf_1}(t)) = - \frac{\partial}{\partial d_{e_0}} J(\bd_{\calf_1})$ and $\frac{d}{dt} J(P_{v_0}(\bd(t))) = -\frac{\partial}{\partial d_{e_0}}J(P_{v_0}(\bd))$ at $t=0$.  Since $d_{e_0}$ is the largest entry of $P_{v_0}(\bd)\in\calAC_{n_{v_0}} - (\calc_{n_{v_0}}\cup\calBC_{n_{v_0}})$, Proposition 1.14 of \cite{DeB_cyclic_geom} implies that the latter derivative is less than $-1/2$.  In the case that $v_1\neq v_T$, that result implies that $\frac{d}{dt} J(\bd_{\calf_1}(t))>0$.  If $v_1 = v_T$ then since $\bd_{\calf_1}\in\calc_{n_1}$ by hypothesis it gives that this derivative is greater than $-1/2$, so in both cases we have the desired inequality.\end{proof}

We now show that $D_T(\bd(t))$ is decreasing for small $t>0$. To do so we use Lemma \ref{many polys} (which applies by construction of $\bd(t)$) to rewrite $\sum_{v\in T_1^{(0)}} D_0(P_v(\bd(t)))$ as $D_0(\bd_{\calf_1}(t))$, yielding:
\begin{align}\label{summa} D_T(\bd(t)) = D_0(\bd_{\calf_1}(t)) + D_0(P_{v_0}(\bd(t))) + \sum_{v\in T^{(0)}-T_0^{(0)}} D_0(P_v(\bd(t))) \end{align}
For any $v\in T^{(0)}$, if the edges of $\cale\cup\calf$ containing $v$ are $e_{i_1},\hdots,e_{i_k}$ (for $k=n_v$) then $P_v(\bd(t)) = (d_{e_{i_1}}(t),\hdots,d_{e_{i_k}}(t))$ and so by the chain rule we have:\begin{align}\label{first time}
  \frac{d}{dt}D_0(P_v(\bd(t))) = \sum_{j=1}^k \frac{\partial}{\partial d_{e_{i_j}}} D_0(P_v(\bd(t)))\,\frac{d}{dt}d_{e_{i_j}}(t) 
\end{align}
If $v\in T^{(0)}-T_0^{(0)}$ this implies in particular that $D_0(P_v(\bd(t)))$ is constant, since $d_e(t)$ is constant for each edge $e$ containing such a vertex $v$.  So the rightmost sum of (\ref{summa}) is constant in $t$.  We now compute $\frac{d}{dt}\left(D_0(\bd_{\calf_1}(t)) + D_0(P_{v_0}(\bd(t)))\right)$ by applying the chain rule as in (\ref{first time}) and Proposition 2.3 of \cite{DeB_cyclic_geom}.  This gives:
$$ \sqrt{\frac{1}{\cosh^2(d_{e_0}/2)} - \frac{1}{\cosh^2 J(P_{v_0}(\bd(t)))}} - \sqrt{\frac{1}{\cosh^2(d_{e_0}/2)} - \frac{1}{\cosh^2 J(\bd_{\calf_1}(t))}}  $$
By \ref{smaller J}, this quantity is negative for $t>0$, so indeed $D_T(\bd(t))$ is decreasing.  The Lemma follows, since by hypothesis $\bd$ is a local minimizer for $D_T$ on $\overline{\mathit{Ad}}(\bd_{\calf})$, and we produced $\bd(t)$ assuming that either $v_T\notin T_0$ or $v_T\in T_0$ but $P_{v_t}\in\calc_{n_T}$.\end{proof}

\newcommand\BDeriv{The function $b_0\co(\mathbb{R}^+)^{n-1}\to \mathbb{R}^+$ defined in Proposition 1.12 of \cite{DeB_cyclic_geom} satisfies $0<\frac{\partial}{\partial d_i} b_0(d_1,\hdots,d_{n-1})<1$ for each $i$.}

\begin{proof}[Proof of Proposition \ref{equal radii}] For $T_0$ as described in the Proposition we have $v_T\in T_0$ and $P_{v_T}(\bd)\in\calBC_{n_T}$, by Lemma \ref{halfway there}.  Let $e_T\in\cale\cup\calf$ be the edge containing $v_T$ that has $d_{e_T}$ maximal among all such edges.  We now suppose by way of contradiction that $e_T\in\cale$, and that its other endpoint $v_T'$ is not on the boundary of $T_0$.  We will show that then $\bd$ is still not a local minimum of $D_T$ on $\overline{\mathit{Ad}}(\bd_{\calf})$.

We begin by observing that Proposition 1.11 of \cite{DeB_cyclic_geom} shows that $v_T'\in T_0$ and $P_{v_T'}(\bd)\in\calBC_{n_T'}$, since $J(P_{v_T}(bd)) = d_{e_T}/2$.  This is because on the one hand, $J(P_{v_T'}(\bd)) \leq J(P_{v_T}(\bd)) = d_{e_T}/2$ by Definition \ref{admissible closure}(3), but on the other $J(P_{v_T'}(\bd))\geq d_{e_T}/2$ since the circumcircle radius of a cyclic polygon is at least half of each of its side lengths (cf.~\cite[Prop.~1.5]{DeB_cyclic_geom}).  

Let $T'$ be the maximal subtree of $T_0$ containing $v_T'$ but not $v_T$, let $v_0'$ be a farthest vertex of $T'$ from $v_T'$, and refer by $e_0'$ to the initial edge $e_{v_0'}$ of the arc joining $v_0'$ to $v_T'$.  Let $T_1' = T' - (\mathit{int}(e_0')\cup v_0')$, and let $\calf_1'$ be its frontier in $V$.  We enumerate $\calf_1'$ as $\{e_0',\hdots,e_{n_1'-2},e_T\}$, so in particular $n_1' = |\calf_1'|$, and define a deformation $\bd_{\calf'_1}(t)$ of the tuple $\bd_{\calf'_1}$ that takes its entries from $\bd$ as follows: take $d_{e_0'}(t) = d_{e_0'}-t$, let $d_{e_i'}(t)\equiv d_{e_i'}$ for all $i>0$, and define $d_{e_T}(t) = b_0(d_{e_0'}(t),d_{e_1'},\hdots,d_{e_{n_1'-2}})$, for $b_0$ from Proposition 1.12 of \cite{DeB_cyclic_geom}.

Note that for the tuple $\bd_1' = (\bd_{\cale'_1},\bd_{\calf_1'})$ that takes its entries from $\bd$, where $\cale_1'$ is the edge set of $T_1'$, we have $P_{v_T'}(\bd_1') = P_{v_T'}(\bd)\in\calBC_{n_T'}$.  If we take $v_T'$ as the root vertex of $T_1'$ it satisfies the hypotheses of Lemma \ref{many polys}, so by that result $\bd_{\calf_1'}\in\calBC_{n_1'}$.  In particular, $d_{e_T} = b_0(d_{e_0'},\hdots,d_{e_{n_1'-1}})$ by \cite[Prop.~1.12]{DeB_cyclic_geom}, so $\bd_{\calf_1'}(0) = \bd_{\calf_1'}$.  We now apply Lemma \ref{deform many polys} to produce a deformation $\bd_{\cale_1'}(t)$ of $\bd_{\cale_1'}$ such that $\bd_1'(t) = (\bd_{\cale_1'}(t),\bd_{\calf_1'}(t))$ satisfies the hypotheses of Lemma \ref{many polys}.

On the other side we define $T_1 = T_0-(T'\cup \mathit{int}(e_T))$, call $\calf_1$ its frontier in $V$ and let   $\bd_{\calf_1}$ be the tuple that takes its entries from $\bd$.  We define a deformation $\bd_{\calf_1}(t)$ of $\bd_{\calf_1}$ by taking $d_{e_T}(t)$ as prescribed above and $d_f(t)\equiv d_f$ for each other edge $f\in\calf_1$.  Let $\cale_1$ be the edge set of $T_1$ and $v_T$ its root vertex.  The tuple $\bd_1 = (\bd_{\cale_1},\bd_{\calf_1})$ taking entries from $\bd$ then satisfies the hypotheses of Lemma \ref{many polys}, inheriting properties (1) and (2) there from $T$ (recalling for (2) that $e_T\in\calf_1$) and (3) from $T_0$.  We thus apply Lemma \ref{deform many polys} to produce $\bd_{\cale_1}(t)$ such that $\bd_1(t) = (\bd_{\cale_1}(t),\bd_{\calf_1}(t))$ satisfies the hypotheses of Lemma \ref{many polys}.

Note that $(\cale_1\cup\calf_1)\cap(\cale'_1\cup\calf'_1) = \{e_T\}$, and we have defined $\bd_1(t)$ and $\bd_1'(t)$ so that their entries corresponding to $e_T$ agree.  We now define $\bd(t) = (\bd_{\cale}(t),\bd_{\calf}(t))$ by taking each $d_e(t)$ from $\bd_1(t)$ if $e\in\cale_1\cup\calf_1$, from $\bd_1'(t)$ if $e\in\cale_1'\cup\calf_1'$, and otherwise letting $d_e(t)\equiv d_e$.  Note that the only entries of $\bd_{\calf_1}(t)$ and $\bd_{\calf_1'}(t)$ that change with $t$ correspond to $e_T$ and $e_0'$, each of which lies in $\cale$, so $\bd_{\calf}(t)\equiv\bd_{\calf}$.  The proof of the Proposition will be completed by showing first that $\bd(t)\in\overline{\mathit{Ad}}(\bd_{\calf})$, then that $D_T(\bd(t))$ is decreasing, for all small enough $t$.

To show that $\bd(t)\in\overline{\mathit{Ad}}(\bd_{\calf})$ we check the criteria of Definition \ref{admissible closure}, beginning with (1).  For all vertices $v$ outside $T_0$, $P_v(\bd(t))\equiv P_v(\bd)\in\calAC_{n_v} - \calc_{n_v}$ has maximum entry $d_{e_v}$ by hypothesis.  For $v\in T_0^{(0)} - \{v_T,v_T'\}$, Lemma \ref{not in BC} asserts that $P_v(\bd)\in \calAC_{n_v} - (\calc_{n_v}\cup\calBC_{n_v})$, so $P_v(\bd(t))$ remains here for small $t>0$ since $\calAC_{n_v} - (\calc_{n_v}\cup\calBC_{n_v})$ is open.  And we chose $\bd_{\calf_1'}(t)\in\calBC_{n'}$ for all $t$, and $\bd_{\cale_1'}(t)$ so that $\bd_1' = (\bd_{\cale_1'}(t),\bd_{\calf_1'}(t))$ satisfies the hypotheses of Lemma \ref{many polys}, so that result implies that $P_{v_T'}(\bd(t))\in\calBC_{n_T'}$ for all $t$.  For each $v\in T^{(0)}-\{v_T\}$, $d_{e_v}$ is the unique maximal entry of $P_v(\bd)$ since it is not in $\calc_{n_v}$ (this follows from \cite[Prop.~2.2]{DeB_cyclic_geom}), so $d_{e_v}(t)$ remains the maximal entry for small $t>0$.

For criterion (2) of Definition \ref{admissible closure} we note first that $\bd_{\calf_1}\in\calBC_{n_1}$ by Lemma \ref{many polys}, where $n_1 = |\calf_1|$, since $P_{v_T}(\bd) = P_{v_T}(\bd_1)\in\calBC_{n_T}$ by hypothesis.  If we enumerate $\calf_1$ as $\{e_T,e_1,\hdots,e_{n_1-1}\}$ then $d_{e_T} > d_{e_i}$ for all $i$, again by Lemma \ref{many polys}, since $d_{e_T}$ is by hypothesis maximal among the $d_e$ for $e$ containing $v_T$.  It therefore follows from Proposition 1.12 of \cite{DeB_cyclic_geom} that $d_{e_T} = b_0(d_{e_1},\hdots,d_{e_{n_1-1}})$ for $b_0$ as defined there.  For $t\geq 0$, our definition of $d_{e_T}(t)$ and the chain rule give $\frac{d}{dt}d_{e_T}(t) = -\frac{\partial b_0}{\partial d_{e_0'}}$.  That this is negative follows from:

\newtheorem*{BDerivLemma}{Lemma \ref{b_0 deriv}}
\begin{BDerivLemma}\BDeriv\end{BDerivLemma}

We will prove Lemma \ref{b_0 deriv} after finishing the current proof.  It implies that for all $t>0$:
$$d_{e_T}(t) < b_0(d_{e_1}(t),\hdots,d_{e_{n_1-1}}(t)) \equiv b_0(d_{e_1},\hdots,d_{e_{n_1-1}})$$
It follows that $\bd_{\calf_1}(t)\in\calc_{n_1}$ (see \cite[Cor.~4.10]{DeB_cyclic_geom}), and hence that $P_{v_T}(\bd(t)) = P_{v_T}(\bd_1(t))\in\calc_{n_T}$, since we constructed $\bd_1(t)$ to satisfy the hypotheses of Lemma \ref{many polys}.

For criterion (3) of Definition \ref{admissible closure} we note that by construction and Lemma \ref{deform many polys}, $J(P_v(\bd(t))) = J(\bd_{\calf_1}(t))$ for all vertices $v$ of $T_1$ and $J(P_v(\bd(t))) = J(\bd_{\calf_1'}(t))$ for $v$ in $T_1'$.  For any vertex $v$ of $T$ that lies outside $T_0$, $P_v(\bd(t)) \equiv P_v(\bd)$ for all $t$, and the strict inequality $J(P_v(\bd))<J(P_w(\bd))$ for any $w\in T_0^{(0)}$ is preserved for small $t$.  The only vertex of $T_0$ that does not lie in $T_1$ or $T_1'$ is $v_0'$, so to check (3) we must only establish that $J(\bd_{\calf_1}(t))\geq J(\bd_{\calf_1'}(t)) \geq J(P_{v_0'}(\bd(t)))$ for all small $t>0$.

We first address $J(\bd_{\calf_1}(t))$.  Applying the chain rule gives:\begin{align}\label{chain rule}
  \frac{d}{dt} J(\bd_{\calf_1}(t)) = \frac{\partial J}{\partial d_{e_T}}\,\frac{d}{dt}d_{e_T}(t) + \sum_{i=0}^{n_1'-2} \frac{\partial J}{\partial d_{e_i}}\,\frac{d}{dt}d_{e_i}(t)\end{align}
By construction, $\frac{d}{dt}d_{e_i}(t) \equiv 0$ for all $i\geq 1$, so the quantity inside the summation above vanishes.  Proposition 1.14 of \cite{DeB_cyclic_geom} implies that if $\bd = (d_0,\hdots,d_{n-1})$ lies in $\calc_n$ then $0<\frac{\partial J}{\partial d_i}(\bd)<\frac{1}{2}$ for all $i$, and if $\bd\in\calBC_n$ then $\frac{\partial J}{\partial d_i}(\bd)=\frac{1}{2}$ if $d_i$ is the largest entry.  Applying this result and the observation above that $\bd_{\calf_1}(0)\in\calBC_{n_1}$ with largest entry $d_{e_T}$, and $\bd_{\calf_1}(t)\in\calc_{n_1}$ for small $t>0$, gives that $\frac{d}{dt} J(\bd_{\calf_1}(0)) = \frac{1}{2}\frac{d}{dt}d_{e_T}(0)$ and $ \frac{1}{2}\frac{d}{dt}d_{e_T}(t)<\frac{d}{dt} J(\bd_{\calf_1}(t)) <0$ for $t>0$.  (Here recall from above Lemma \ref{b_0 deriv} that $\frac{d}{dt}d_{e_T}(t) < 0$ for all $t$.)

Applying the chain rule to $\frac{d}{dt}J(\bd_{\calf_1'}(t))$ in the same way as in (\ref{chain rule}) gives $\frac{d}{dt} J(\bd_{\calf_1}(t)) = \frac{\partial J}{\partial d_{e_T}}\frac{d}{dt}d_{e_T}(t) + \frac{\partial J}{\partial d_{e_0'}}\,\frac{d}{dt}d_{e_0'}(t)$.  We chose $d_{e_T}(t)$ so that $\bd_{\calf_1'}(t)\in\calBC_{n_1'}$ for all $t$, so \cite[Prop.~1.14]{DeB_cyclic_geom} implies that $\frac{\partial J}{\partial d_{e_0'}} \equiv 0$ and $\frac{\partial J}{\partial d_{e_T}} \equiv \frac{1}{2}$.  (These equalities follow from the inequalities recorded in that result by continuity of the partial derivatives of $J$, recalling from \cite{DeB_cyclic_geom} that $\calBC_n$ is the frontier of $\calc_n$ in $\calAC_n$ for any $n$, and from above that $d_{e_T}(t)$ is the largest entry of $\bd_{\calf_1}(t)$ for small $t\geq 0$.)  It follows that $\frac{d}{dt}J(\bd_{\calf_1'}(t)) = \frac{1}{2}\frac{d}{dt}d_{e_T}(t)$ for all $t$, hence that $J(\bd_{\calf_1}(t)) \geq J(\bd_{\calf_1'}(t))$ for all $t\geq 0$.

We finally compute $\frac{d}{dt}J(P_{v_0'}(\bd(t)))$ as in (\ref{chain rule}), yielding $\frac{\partial J}{\partial d_{e_0'}}\,\frac{d}{dt}d_{e_0'}(t) = -\frac{\partial J}{\partial d_{e_0'}}$.  We established above that $P_{v_0'}(\bd(t))$ remains in $\calAC_{n_{v_0'}} - (\calc_{n_{v_0'}}\cup\calBC_{n_{v_0'}})$ for all small $t\geq 0$, so \cite[Prop.~1.14]{DeB_cyclic_geom} implies that $\frac{\partial J}{\partial d_{e_0'}} > \frac{1}{2}$ for all such $t$.  Therefore $\frac{d}{dt}J(P_{v_0'}(\bd(t)) < -\frac{1}{2}$ for all $t$.  But Lemma \ref{b_0 deriv} implies that $\frac{d}{dt}d_{e_T}(t) = -\frac{\partial b_0}{\partial d_{e_0'}} > -1$, so $\frac{d}{dt}J(\bd_{\calf_1'}(t)) > -\frac{1}{2}$ for all $t\geq 0$.  It follows that $J(\bd_{\calf_1'}(t)) \geq J(P_{v_0'}(\bd(t)))$ for all $t\geq 0$ as claimed, and hence that $\bd(t)\in\overline{\mathit{Ad}}(\bd_{\calf})$ for such $t$.

To show that $D_T(\bd(t))$ is decreasing we use the following consequence of our construction and Lemma \ref{many polys}:
$$ D_T(\bd(t)) = D_0(P_{v_0'}(\bd(t))) + D_0(\bd_1'(t)) + D_0(\bd_1(t)) + \sum_{v\in T^{(0)}-T_0^{(0)}} D_0(P_v(\bd(t))) $$
We now apply the chain rule and \cite[Prop.~2.3]{DeB_cyclic_geom} to compute $\frac{d}{dt} D_T(\bd(t))$, yielding:
\begin{align}\label{D_T deriv}
   -\left[\sqrt{\frac{1}{\cosh^2(d_{e_0'}(t)/2)} - \frac{1}{\cosh^2 J(\bd_{\calf_1'}(t))}} -\sqrt{\frac{1}{\cosh^2(d_{e_0'}(t)/2)} - \frac{1}{\cosh^2 J(P_{v_0'}(\bd(t)))}}\right]\quad\nonumber\\
   + \frac{d}{dt}d_{e_T}(t)\sqrt{\frac{1}{\cosh^2(d_{e_T}(t)/2)} - \frac{1}{\cosh^2 J(\bd_{\calf_1}(t))}} \end{align}
Here we are using the fact that only $d_{e_0'}(t)$ and $d_{e_T}(t)$ are non-constant among all entries of $\bd(t)$; that $d_{e_0'}(t)$ is an entry of $P_{v_0'}(\bd(t))$ and $\bd_{\calf_1'}(t)$, and that $\frac{d}{dt}d_{e_0'}(t) = -1$.  This yields the top line above.  For the bottom line we use that $d_{e_T}(t)$ is an entry of $\bd_{\calf_1'}(t)$ and $\bd_{\calf_1}(t)$, and that $J(\bd_{\calf_1'}(t)) \equiv d_{e_T}/2$, which implies that $\frac{\partial}{\partial d_{e_T}}D_0(\bd_{\calf_1'}(t)) \equiv 0$.

The derivative recorded above is $0$ at $t=0$, since there all circumcircle radii are equal to $d_{e_T}/2$.  But it is negative for small $t>0$, since we showed above that $J(\bd_{\calf_1'}(t)) > J(P_{v_0'}(\bd(t)))$ and $\frac{d}{dt}d_{e_T}(t) < 0$ for such $t$.  It follows that $D_T(\bd(t))$ decreases with $t$, and we have proved that if $\bd\in\overline{\mathit{Ad}}(\bd_{\calf})$ is a local minimizer for $\bd\mapsto D_T(\bd)$ such that $J(P_v(\bd)) = J(P_w(\bd))$ for some $v\in T^{(0)}$ and $w\in v-1$ then $v_T\in T_0$, $P_{v_T}(\bd)\in\calBC_{n_T}$, and either $e_T\in\calf$ or its other endpoint $v_T'$ lies on the boundary of $T_0$.

We have proved the third bulleted assertion of Proposition \ref{equal radii}.  It remains to show for $\bd$ as above that if $e_T\in\cale$ then $P_v(\bd)\in\calBC_{n_v}$ for each $v\in v_T'-1$.  Suppose not; ie.~that there exists $v_0'\in v_T'-1$ such that $P_{v_0'}(\bd)\in\calAC_{n_{v_0'}} - (\calc_{n_{v_0'}}\cup\calBC_{n_{v_0'}})$.  Let $e_0'$ denote $e_{v_0'}$, and take $T_1' = \{v_T'\}$ and $T_1 = T_0-(T_1'\cup \mathit{int}(e_T))$.  The frontier $\bd_{\calf_1'}$ of $T_1'$ in $V$ is the set of edges containing $v_T'$, and we take $\bd_{\calf_1'} = P_{v_T'}(\bd)$.  Letting $\cale_1$ and $\calf_1$ denote the edge set and frontier in $V$ of $T_1$, respectively, we define $\bd_{\calf_1'}(t)$, $\bd_{\calf_1}(t)$, $\bd_{\cale_1}(t)$ and $\bd(t)$ exactly as in the previous case (note that here $\cale_1'=\emptyset$).  That is, we let $d_{e_0'}(t) = d_{e_0'}-t$; choose $d_{e_T}(t)$ as before so that $\bd_{\calf_1'}(t)\in\calBC_{n_1'}(t)$ for all $t$, where $n_1' = |\calf_1'|$; let Lemma \ref{deform many polys} determine $\bd_{\cale_1}(t)$; and let $d_e(t)\equiv d_e$ for all $e\in (\cale\cup\calf)- (\calf_1'\cup\cale_1\cup\calf_1)$.

The same argument as in the previous case now shows that $\bd(t)\in\overline{\mathit{Ad}}(\bd_{\calf})$ for small enough $t>0$, and $D_T(\bd(t))$ decreases in $t$, with only a couple slight modifications.  We first note that it is still true that $J(P_{v_0'}(\bd)) > J(P_v(\bd))$ for all $v \in v_0'-1$, since we have already showed that if not then $v_T$ is in the maximal subtree containing $v_0'$ with all vertices $v$ satisfying $J(P_v(\bd)) = J(P_{v_0'}(\bd))$.  This fact is necessary for showing that $\bd(t)\in\overline{\mathit{Ad}}(\bd_{\calf})$ for small enough $t>0$.  And the computation of $\frac{d}{dt}D_T(\bd(t))$ is identical, but in this case the derivative is negative at $t=0$ since the fact that $v_0'$ is not in $T_0$ implies that $J(P_{v_0'}(\bd)) < J(\bd_{\calf_1'})$ (in the previous case equality held).  But this only helps us, and the result follows.\end{proof}

\begin{lemma}\label{b_0 deriv}\BDeriv\end{lemma}

\begin{proof}[Proof of Lemma \ref{b_0 deriv}] The proof of \cite[Prop.~1.12]{DeB_cyclic_geom} establishes the inequality $\frac{\partial b_0}{\partial d_i} > 0$.  There the first of the following equations is showed:\begin{align*}
  \frac{\partial}{\partial d_i} b_0(d_1,\hdots,d_{n-1}) = - \frac{\frac{\partial \theta}{\partial d}(d_i,b_0/2)}{\sum_{j=1}^{n-1} \frac{\partial\theta}{\partial J}(d_i,b_0/2)} = \frac{\cosh(d_i/2)\sinh(b_0/2)}{\cosh (b_0/2)\sum_j \sinh(d_j/2)} \end{align*}
Here $\theta(d,J)$ is the function described in Lemma 1.4 of \cite{DeB_cyclic_geom}, that measures the angle of an isosceles triangle with two sides of length $J$ and one of length $d$ at its vertex opposite the side of length $d$.  The latter equation above follows by simply computing partial derivatives.

We recall that $b_0 > d_i$ for each $i$, by \cite[Prop~1.12]{DeB_cyclic_geom}.  The result now follows by observing that $\sum_j \sinh(d_j/2) > \sinh(b_0/2)$, since by \cite[Prop.~1.11]{DeB_cyclic_geom}, $\calBC_n\subset \calAC_n$.  \end{proof}

\begin{corollary}\label{update}  Let $T\subset V$ be a compact rooted tree with root vertex $\{v_T\}$, frontier $\calf = \{f_0,\hdots,f_{n-1}\}$, and edge set $\cale$.  For $\bd_{\calf} = (d_{f_0},\hdots,d_{f_{n-1}})\in(\mathbb{R}^+)^{\calf}$ such that $\overline{\mathit{Ad}}(\bd_{\calf})\neq\emptyset$, at a point $\bd = (\bd_{\cale},\bd_{\calf})\in\overline{\mathit{Ad}}(\bd_{\calf})$ which is a local minimum of the map $\bd\mapsto D_T(\bd_{\cale},\bd_{\calf})$, either $P_v(\bd)\in\calBC_{n_v}$ for each $v\in T^{(0)} - \{v_T\}$ or the following hold.

$P_{v_T}(\bd)\in\calBC_{n_T}$, where $v_T$ has valence $n_T$ in $V$, and for the edge $e_T\in\cale\cup\calf$ containing $v_T$ such that $d_{e_T}$ is maximal among all such edges, either $e_T\in\calf$ or $e_T$ is an edge of the maximal subtree $T_0$ containing $v_T$ with the property that $J(P_v(\bd)) = J(P_{v_T}(\bd))$ for all $v\in T_0^{(0)}$, with its other endpoint $v_T'$ on the boundary of $T_0$.  Moreover, for every vertex $v$ of $T$ such that $P_v(\bd)\notin\calBC_{n_v}$, $e_v$ has its other endpoint in $T_0$.  In the case that $e_T\in\cale$, there is no such $v\in v_T'-1$.\end{corollary}

\begin{proof}  Proposition 3.23 of \cite{DeB_Voronoi} asserts that any local minimizer $\bd$ for $\bd\mapsto D_T(\bd)$ on $\overline{\mathit{Ad}}(\bd_{\calf})$ satisfies one of three criteria that it lists.  (The result is only stated there for absolute minimizers, but inspecting its proof shows that it holds for local minimizers.)  Criterion (1) is the condition that $P_v(\bd)\in\calBC_{n_v}$ for all $v\in T^{(0)}-\{v_T\}$.  We thus suppose now that $\bd$ is a local minimum that does not satisfy (1).

Criterion (2) of \cite[Prop.~3.23]{DeB_Voronoi} is that $P_{v_T}\in\calBC_{n_T}$, and criterion (3) is that $J(P_v(\bd)) = J(P_w(\bd))$ for some $v\in T^{(0)}$ and $w\in v-1$.  But Proposition \ref{equal radii} implies that if $\bd$ satisfies criterion (3) then it also satisfies (2).  In fact, defining $T_0$ as we have here, Proposition \ref{equal radii} implies that $T_0$ contains every $v\in T^{(0)}$ such that $J_v(\bd) = J_w(\bd)$ for some $w\in T^{(0)}$ that is joined to $v$ by an edge of $T$; hence that $J_v(\bd) = J_{v_T}(\bd)$ for all such $v$.

For every vertex $v$ of $T$ outside $T_0$, it follows that the inequality of Definition \ref{admissible closure}(3) is strict; ie.~that $J(P_v(\bd)) > J(P_w(\bd))$ for all $w\in v-1$, and moreover, that $J(P_v(\bd)) < J(P_{v'}(\bd))$ for such $v$, where $v'$ is the other endpoint of $e_v$.  If $P_v(\bd)\notin\calBC_{n_v}$ and $v'\notin T_0$, then we claim that $d_{e_v}$ can be decreased slightly, keeping all other entries of $\bd$ constant, to produce $\bd'\in\overline{\mathit{Ad}}(\bd_{\calf})$, with $D_T(\bd')<D_T(\bd)$.

The key effects of the deformation are that $J(P_v(\bd')) < J(P_v(\bd))$ and $J(P_{v'}(\bd')) > J(P_{v'}(\bd))$, whereas $J(P_w(\bd')) = J(P_w(\bd))$ for all other $w\in T^{(0)}$.  But since we assumed that $v'\notin T_0$, if $w$ is the other endpoint of $e_{v'}$ then $J(P_w(\bd))>J(P_{v'}(\bd))$, so this inequality is preserved by choosing $d_{e_v}'$ near enough to $d_{e_v}$.  Similarly, if $w\in v-1$ the strict inequality $J(P_w(\bd))<J(P_v(\bd))$ is preserved upon choosing $d_{e_v}'$ near enough to $d_{e_v}$.  Arguing as in the next-to-last paragraph of the proof of \cite[Prop.~3.23]{DeB_Voronoi} shows that $P_{v'}(\bd')$ remains in $\calAC_{n_{v'}} - \calc_{n_{v'}}$, and since $P_v(\bd)$ lies in the open set $\calAC_{n_v} - (\calc_{n_v}\cup\calBC_{n_v})$ by hypothesis, it remains there for $d_{e_v}'$ near enough to $d_{e_v}$.  The fact that $D_T(\bd')<D_T(\bd)$ follows as in the second paragraph of the proof of \cite[Prop.~3.23]{DeB_Voronoi}, and the claim is proved.

The claim implies for all $v\in T^{(0)}$ such that $P_v(\bd)\notin\calBC_{n_v}$ that $e_v$ has its other endpoint in $T_0$.  The remaining properties of $T_0$ follow from Proposition \ref{equal radii}.  We note that as defined here we could have $T_0 = \{v_T\}$, if $\bd$ satisfies criterion (2) of \cite[Prop.~3.23]{DeB_Voronoi} but not (3), but in this case $e_T$ must lie in $\calf$.  This is because the fact that $P_{v_T}(\bd)\in\calBC_{n_T}$ implies that $J(P_{v_T}(\bd)) = d_{e_T}/2$ \cite[Prop.~1.11]{DeB_cyclic_geom}, so if $e_T\in\cale$ then the fact that its other endpoint $v_T'$ must satisfy $J(P_{v_T'}(\bd))\leq J(P_{v_T}(\bd))$ (by Definition \ref{admissible closure}(1)) but $J(P_{v_T'}(\bd))\geq d_{e_T}/2$ \cite[Prop.~1.5]{DeB_cyclic_geom} implies that $J(P_{v_T'}(\bd) = d_{e_T}/2 = J(P_{v_T}(\bd))$.\end{proof}

\subsection{Allowing $\bd_{\calf}$ and $T$ to vary}\label{alla the marbles}  The first main result of this subsection, Proposition \ref{centered lower}, generalizes and strengthens Proposition 3.30 of \cite{DeB_Voronoi}.  The idea here is to bound $D_T(\bd)$ below for a fixed compact, rooted tree $T$ with frontier $\calf$, but with $\bd_{\calf}$ allowed to vary with its entries bounded below by those of some fixed $\bfb_{\calf}\in(\mathbb{R}^+)^{\calf}$.  After this we prove Lemma \ref{trivalent}, which compares minima of $D_T(\bd)$ for different trees $T$, then prove Theorem \ref{upgrade}.

\begin{proposition}\label{centered lower}  Suppose $T\subset V$ is a compact, rooted tree with root vertex $v_T$, edge set $\cale$, and frontier $\calf$. Given $\bfb_{\calf}\in(\mathbb{R}^+)^{\calf}$, define $\bfb_{\cale}\in(\mathbb{R}^+)^{\cale}$ by $b_e = b_e(\bfb_{\calf})$ for each $e\in\cale$, where $b_e\co(\mathbb{R}^+)^{\calf}\to\mathbb{R}^+$ is as in \cite[Lemma 3.19]{DeB_Voronoi}.  Enumerating the edges of $\cale\cup\calf$ containing $v_T$ as $e_0,\hdots,e_{n_T-1}$ so that $b_{e_0}$ is maximal, define $m_{e_0} = \min\{b_{e_0},b_0(b_{e_1},\hdots,b_{e_{n_T-1}})\}$, and take:
$$ B_T(\bfb_{\calf}) = D_0(m_{e_0},b_{e_1},\hdots,b_{e_{n_T-1}}) + \sum_{v\in T^{(0)} - \{v_T\}} D_0(P_v(\bfb_{\cale},\bfb_{\calf})) $$
Then for each $\bd_{\calf}\in(\mathbb{R}^+)^{\calf}$ such that $d_f\geq b_f$ for each $f\in\calf- \{e_0\}$ and $d_{e_0}\geq m_{e_0}$, and each $\bd\in\overline{\mathit{Ad}}(\bd_{\calf})$, $D_T(\bd)\geq B_T(\bfb_{\calf})$.\end{proposition}

If $e_0\notin\calf$ then the requirement above on $\bd_{\calf}$ simply becomes that $d_f\geq b_f$ for all $f\in\calf$.  In the case that $e_0\in\calf$ we note that the given bound is \textit{a priori} stronger than one which holds for all $\bd_{\calf}$ with $d_f\geq b_f$ for all $f\in\calf$, since $m_{e_0}\leq b_{e_0}$.

The proof proceeds by separately considering the possibilities described in Corollary \ref{update} for minimizers of the function $\bd\mapsto D_T(\bd)$.  One is handled by the lemma below.

\begin{lemma}\label{case 1} Suppose $T\subset V$ is a rooted tree with root vertex $v_T$, edge set $\cale$, and frontier $\calf$, and fix $\bfb_{\calf}\in(\mathbb{R}^+)^{\calf}$.  Suppose for $\bd_{\calf}\in(\mathbb{R}^+)^{\calf}$ and $\bd = (\bd_{\cale},\bd_{\calf})\in\overline{\mathit{Ad}}(\bd_{\calf})$, that\begin{enumerate}
   \item $P_v(\bd)\in\calBC_{n_v}$ for each $v\in T^{(0)}-\{v_T\}$; and
   \item $d_{f} \geq b_{f}$ for each $f\in\calf-\{e_0\}$, and $d_{e_0}\geq m_{e_0}$, for $e_0$ and $m_{e_0}$ as in Proposition \ref{centered lower}.\end{enumerate}
Then $D_T(\bd) \geq B_{T}(\bfb_{\calf})$, for $B_T$ as defined in Proposition \ref{centered lower}.\end{lemma}
  
\begin{proof}  For a given $\bd$, hypothesis (1) above and the defining property of the functions $b_e$ from \cite[Lemma 3.19]{DeB_Voronoi} imply for each $e\in\cale$ that $d_e = b_e(\bd_{\calf})$.  The monotonicity property of $b_e$ laid out in assertion (3) of \cite[L.~3.19]{DeB_Voronoi} and the hypothesis that $d_f\geq b_f$ for each $f\in\calf$ thus together imply that $d_e \geq b_e$ for each $e\in\cale$.  Corollary 2.4 of \cite{DeB_cyclic_geom} now directly gives for each $v\in T^{(0)}-\{v_T\}$ that $D_0(P_v(\bfb_{\cale},\bfb_{\calf}))\leq D_0(P_v(\bd))$.  

For $e_0,\hdots,e_{n_T-1}$ as in Proposition \ref{centered lower}, with $m_{e_0}$ as defined there we note that by construction $(m_{e_0},d_{e_1},\hdots,d_{e_{n_T-1}})$ lies in $\calc_{n_T}\cup\calBC_{n_T}$.  Therefore since $m_{e_0}\leq b_{e_0}$ and $P_{v_T}(\bd)\in\calc_{n_T}\cup\calBC_{n_T}$ by Definition \ref{admissible closure}(2), \cite[Cor.~2.4]{DeB_cyclic_geom} also implies that $D_0(m_{e_0},b_{e_1},\hdots,b_{e_{n_T-1}})\leq D_0(P_{v_T}(\bd))$.  The Lemma therefore follows from the definitions of $B_T(\bfb_{\calf})$ and $D_T(\bd_{\calf})$.\end{proof}

\begin{proof}[Proof of Proposition \ref{centered lower}]  For $\bfb_{\calf}\in(\mathbb{R}^+)^{\calf}$ and $T$ as in the Proposition, compute $\bfb_{\cale}$ and $m_{e_0}$ as prescribed there.  Now enumerate $\calf$ as $\{f_1,\hdots,f_n\}$, where $f_n = e_0$ if $e_0\in\calf$.  For the set $\mathit{SAd}_T=\{\bd_{\calf}\in(\mathbb{R}^+)^{\calf}\,|\,\overline{\mathit{Ad}}(\bd_{\calf})\neq\emptyset\}$ defined in Lemma 3.29 of \cite{DeB_Voronoi}, that result implies that the set below is closed in $\mathbb{R}^n$.
$$ \mathit{SAd}_T(\bfb_{\calf}) = \mathit{SAd}_T \cap \{(d_1,\hdots,d_n)\in\mathbb{R}^n\,|\,d_i\geq b_{f_i}\ \mbox{for each}\ i<n,\ \mbox{and}\ d_n\geq m_{e_0}\ \mbox{or}\ b_{f_n}\} $$
Here the inequality $d_n\geq m_{e_0}$ applies if $e_0\in \calf$, hence $f_n = e_0$; otherwise we require $d_n\geq b_{f_n}$.  

Since $\mathit{SAd}_T(\bfb_{\calf})$ is closed its intersection with $[0,D]^n$ is compact for any fixed $D>0$.  Let $D$ be large enough that $\mathit{SAd}_T(\bfb_{\calf})\cap [0,D]^n$ is non-empty.  (It is not hard to show that such a $D$ exists.)  Then \cite[L.~3.29]{DeB_Voronoi} further implies that the function $\bd_{\calf}\mapsto \min\{D_T(\bd)\,|\,\bd\in\overline{\mathit{Ad}}_T(\bd_{\calf})\}$ attains a minimum on it, since it asserts that this function is lower-semicontinuous on $\mathit{SAd}_T$.

Let $\bd_{\calf}$ be a minimizer for $\min\{D_T(\bd)\,|\,\bd\in\overline{\mathit{Ad}}(\bd_{\calf})\}$ on $\mathit{SAd}_T(\bfb_{\calf})\cap[0,D]^n$, and $\bd$ a minimizer for $D_T(\bd)$ on $\overline{\mathit{Ad}}(\bd_{\calf})$.  We apply Corollary \ref{update} to $\bd$, separately treating the different cases it describes.  If $P_v(\bd)\in\calBC_{n_v}$ for all $v\in T^{(0)}-\{v_T\}$ then Lemma \ref{case 1} directly implies the desired bound.  So we will assume now that $P_v(\bd)\notin\calBC_{n_v}$ for some $v\in T^{(0)}-\{v_T\}$, and therefore by Cor.~\ref{update} that $P_{v_T}(\bd)\in\calBC_{n_T}$, where $v_T$ has valence $n_T$ as in the Corollary.

Let $e_T$ be the edge containing $v_T$ such that $d_{e_T}$ is maximal among all such edges.  We first suppose that $e_T\in\calf$.  In this case we will deform $\bd_{\calf}$ within $\mathit{SAd}_T(\bfb_{\calf})\cap [0,D]^n$ to reduce $D_T(\bd)$, thereby contradicting our minimality hypothesis.  The idea is to reduce $d_{e_T}$ without changing any other entry of $\bd_{\calf}$, so we first show there is room to do this in $\mathit{SAd}_T(\bfb_{\calf})$.  With the edges containing $v_T$ numbered $e_0,\hdots,e_{n_T-1}$ so that $b_{e_0}$ is maximal, we have $e_T = e_{i_0}$ for some $i_0$.  We claim that if $i_0\ne 0$ then $d_{e_T}>b_{e_T} = b_{e_{i_0}}$, and if $i_0 = 0$ then $d_{e_T} > m_{e_0}$.

(The strict inequalities are the point here, since the hypothesis that $d_{f_i}\geq b_{f_i}$ for each $f_i\in\calf-\{e_T\}$ implies that $d_e \geq b_e$ for each $e\in\cale$, by assertions (2) and (3) of \cite[L.~3.19]{DeB_Voronoi}.  This is not clear directly from that lemma's statement, since we do not necessarily have $d_{e_T}\geq b_{e_T}$ here, but note that the first sentence of its proof asserts for each $v\in T^{(0)}-\{v_T\}$ that $b_{e_v}(\bd_{\calf})$ is determined by $\bd_{\calf}$ and the collection of $b_{e_w}(\bd_{\calf})$ for $w<v$.  In fact the proof itself shows more precisely that $b_{e_v}(\bd_{\calf})$ is determined by the set of $d_{f_i}$ such that $f_i\in\calf$ contains $v$ or some $w<v$.  Here $e_T$ contains only $v_T$, and of course $v_T\not<v$ for any $v$.)

First suppose that $i_0\neq 0$.  By the above $d_{e_i} \geq b_{e_i}$ for each $i\neq i_0$, and in particular $d_{e_0}\geq b_{e_0}$.  The claim now follows from the fact that since $P_{v_T}(\bd)\in\calBC_{n_T}$,
\[ d_{e_T} = b_0(d_{e_0},\hdots,\widehat{d_{e_{i_0}}},\hdots,d_{e_{n_T-1}}), \]
so recalling \cite[Prop.~1.12]{DeB_cyclic_geom} we have that it is larger than $d_{e_0}\geq b_{e_0}$.

Now suppose that $i_0 = 0$.  In this case we use our assumption that $P_v(\bd)\notin\calBC_{n_v}$ for some $v\in T^{(0)}-\{v_T\}$, which implies for the initial edge $e_v$ of the arc of $T$ joining $v$ to $v_T$ that
\[ d_{e_v} > b_0(d_{e_1'},\hdots,d_{e_{n_v-1}'}) \geq b_{e_v} = b_0(b_{e_1'},\hdots,b_{e_{n_v-1}'}), \]
where $e_1',\hdots,e_{n_v-1}'$ are the other edges of $\cale\cup\calf$ containing $v$.  It now follows from Lemma \ref{b_0 deriv} that strict inequality $d_e>b_e$ holds for each edge $e$ on the arc joining $v$ to $v_T$, using $e_v$ as the base case of an inductive argument.  This holds in particular for one of the edges $e_i$ containing $v_T$, for $i>0$, so since $P_{v_T}(\bd)\in\calBC_{n_T}$ we again obtain from Lemma \ref{b_0 deriv} that
\[ d_{e_T} = b_0(d_{e_1},\hdots,d_{e_{n_T-1}}) > b_0(b_{e_1},\hdots,b_{e_{n_T-1}}). \]
It thus follows that $d_{e_T}>m_{e_0}$, and the claim is proved in this case as well.

We now define $\bd_{\calf}(t)$ by taking $d_{e_T}(t) = d_{e_T}-t$ and $d_{f_i}(t)\equiv d_{f_i}$ for all $f_i\in\calf-\{e_T\}$.  We will next produce a deformation $\bd(t)\in\overline{\mathit{Ad}}(\bd_{\calf}(t))$ of $\bd$ for small $t>0$, from which (together with the claim) it will follow that $\bd_{\calf}(t)$ is a deformation of $\bd_{\calf}$ within $\mathit{SAd}_T(\bfb_{\calf})\cap [0,D]^n$.  We will further observe that  $D_T(\bd(t))$ decreases with $t$, thus obtaining a contradiction to the hypothesis that $\bd_{\calf}$ is a minimizer for $\min\{D_T(\bd)\,|\,\bd\in\overline{\mathit{Ad}}(\bd_{\calf})\}$ on $\mathit{SAd}_T(\bfb_{\calf})\cap[0,D]^n$.

Let $T_0$ be the maximal subtree of $T$ such that $J(P_v(\bd)) = J(P_{v_T}(\bd))$ for all $v\in T_0^{(0)}$.  Define $\bd(t)$ by letting $d_e(t) \equiv d_e$ for each edge $e$ of $T$ that does not lie in $T_0$, and for $e\subset T_0$ let $d_e(t)$ be determined by Lemma \ref{deform many polys}.  The verification that $\bd(t)\in\overline{\mathit{Ad}}(\bd_{\calf}(t))$ parallels the corresponding check for the deformation $\bd(t)$ described in the proof of the third bulleted assertion of Proposition \ref{equal radii}, but it is simpler.  In particular, $T_0$ here plays the role of $T_1$ there, but there is no $T_1'$ or $v_0'$, and here we have $\frac{d}{dt} d_{e_T}(t) \equiv -1$ rather than $\frac{d}{dt} d_{e_T}(t) = - \frac{\partial b_0}{\partial d_{e_0'}}<0$ as there (see above the statement of Lemma \ref{b_0 deriv}).  As was the case with $\bd_{\calf_1}(t)$, we have $\bd_{\calf_0}(t)\in\calc_{n_0}$ for small $t>0$, where $\calf_0$ is the frontier of $T_0$ and $n_0=|\calf_0|$.  We thus obtain
\[ \frac{d}{dt}D_T(\bd(t)) = -\sqrt{\frac{1}{\cosh^2(d_{e_T}(t)/2)} - \frac{1}{\cosh^2 J(\bd_{\calf_0}(t))}} < 0 \]
for small $t>0$.  This is the analog of (\ref{D_T deriv}) in the proof of Proposition \ref{equal radii}, but again it is simpler since there is no $T_1'$ or $v_0'$, and here $e_T\in\calf$ instead.  It implies that $D_T(\bd(t))$ is decreasing as asserted, finishing the case that $P_{v_T}(\bd)\in\calBC_{n_T}$ and $e_T\in\calf$.

We finally address the case that $P_{v_T}(\bd)\in\calBC_{n_T}$ and $e_T\in\cale$.  Here we will argue by induction on the number of edges of $T$, ie.~$|\cale|$.  The base case $T = \{v_T\}$, with $\cale = \emptyset$, follows directly from Lemma \ref{case 1}, since criterion (1) there holds vacuously (here recall Remark \ref{too good}.)  So assume now that $T$ has $k\geq 1$ edges, that the Proposition holds for all trees with fewer than $k$ edges, and that $P_{v_T}(\bd)\in\calBC_{n_T}$ with $e_T\in\cale$.

Let $T'$ be the maximal subtree containing the other endpoint $v_T'$ of $e_T$ but not $v_T$, and take $T'' = T-(T'\cup\mathit{int}(e_T))$.  We take $v_T'$ as the root vertex of $T'$ and $v_T$ as the root vertex of $T''$.  Naming the frontiers of $T'$ and $T''$ as $\calf'$ and $\calf''$, respectively, we have $\calf'\cap\calf'' = \{e_T\}$ and $\calf'\cup\calf'' = \calf\cup\{e_T\}$.  Similarly taking their edge sets to be $\cale'$ and $\cale''$, we have $\cale'\cap\cale'' = \emptyset$ and $\cale'\cup\cale''\cup\{e_T\}=\cale$.  Thus by the induction hypothesis the Proposition holds for $T'$ and $T''$.

Let $\bd' = (\bd'_{\cale'},\bd'_{\calf'})$ and $\bd'' = (\bd''_{\cale''},\bd''_{\calf''})$ take their entries from $\bd$, so in particular $d_{e_T}' = d_{e_T}'' = d_{e_T}$.  Then for any vertex $v$ of $T'$, $P_v(\bd') = P_v(\bd)$, and similarly if $v\in T''$.  Therefore $\bd'$ lies in $\overline{\mathit{Ad}}(\bd'_{\calf'})$, and $\bd''\in\overline{\mathit{Ad}}(\bd''_{\calf''})$: criteria (1) - (3) of Definition \ref{admissible closure} are directly inherited by $T''$ from $T$, and for $T'$ we merely note in addition that $P_{v_T'}(\bd')\in\calBC_{n_T'}$ by hypothesis, and $J(P_{v_T'}(\bd')) = J(P_{v_T}(\bd)) = d_{e_T}/2$.  Note that $D_T(\bd) = D_{T'}(\bd')+D_{T''}(\bd'')$, since each vertex of $T$ lies in exactly one of $T'$ or $T''$.

Define $\bfb'_{\calf'}$ by taking $b_{e_T}' = b_{e_T} = b_{e_T}(\bfb_{\calf})$ and pulling the remaining entries from $\bfb_{\calf}$, and define $\bfb''_{\calf''}$ by taking all entries but $b''_{e_T}$ from $\bfb_{\calf}$.  To obtain $b''_{e_T}$ we enumerate the edges containing $v_T$ as $e_0,\hdots,e_{n_T-1}$ as described in the Proposition, for each $i$ let $b_{e_i}$ be the appropriate entry of $\bfb_{\calf}$ or $\bfb_{\cale}$, and for $i_0$ such that $e_T = e_{i_0}$ we take
$$b''_{e_T} = b_0(b_{e_0},\hdots,\widehat{b_{e_{i_0}}},\hdots,b_{e_{n_T}-1})$$
We will establish this case of the Proposition with two claims: first, that $B_{T'}(\bd'_{\calf'})+B_{T''}(\bd''_{\calf''}) \geq B_T(\bd_{\calf})$, and second, that $d'_e \geq b'_e$ for all $e\in\calf'$ and $d''_e\geq b''_e$ for all $e\in\calf''$.  Applying induction and the observation that $D_T(\bd) = D_{T'}(\bd') + D_{T''}(\bd'')$, we therefore conclude the desired bound $D_T(\bd)\geq B_T(\bfb_{\calf})$.

Toward the first claim, computing straight from the definitions gives:
$$ B_{T'}(\bd'_{\calf'})+B_{T''}(\bd''_{\calf''}) - B_T(\bd_{\calf}) = D_0(b_{e_0},\hdots,b''_{e_T},\hdots,b_{e_{n_T-1}}) - D_0(m_{e_0},b_{e_1},\hdots,b_{e_{n_T-1}}) $$
The entries from the two inputs to $D_0$ on the right-hand side above differ only in the $e_0$ and $e_{i_0}$ positions.  If $i_0 = 0$ then it follows directly from their definitions that $m_{e_0} \leq b''_{e_T}$.  Otherwise, by its definition in the Proposition we have $m_{e_0}\leq b_{e_0}$.  And since $b_{e_0}$ is maximal among the $b_{e_i}$ by definition, $b_{e_T}'' > b_{e_0} \geq b_{e_{i_0}}$ by its definition and Proposition 1.12 of \cite{DeB_cyclic_geom}.  In either case the difference above is positive by \cite[Cor.~2.4]{DeB_cyclic_geom}, yielding the first claim.

For the second claim we note first that by hypothesis $d_e \geq b_e$ for each $e\in\calf$, so by definition $d_e' = d_e \geq b_e$ for all $e\in\calf'-\{e_T\}$, and similarly for $e\in\calf''-\{e_T\}$.  Applying Lemma 3.19 of \cite{DeB_Voronoi} gives 
$$d_e \geq b_e(\bd_{\calf}) \geq b_e(\bfb_{\calf}) = b_e$$
for each $e\in\cale$.  (The first and second inequalities above are respectively implied by assertions (2) and (3) there.)  It now follows immediately that $d_{e_T}' = d_{e_T} \geq b_{e_T}' = b_{e_T}$, so the claim is proved for $\bd_{\calf'}$.  Since $P_{v_T}(\bd)\in\calc_{n_T}$, Proposition 1.12 of \cite{DeB_cyclic_geom} implies that
$$ d_{e_T} = b_0(d_{e_0},\hdots,\widehat{d_{e_i}},\hdots,d_{e_{n_T-1}}) \geq b_0(b_{e_0},\hdots,\widehat{b_{e_i}},\hdots,b_{e_{n_T-1}}) $$
But the latter quantity is $b_{e_T}''$, and since $d_{e_T}'' = d_{e_T}$ the second claim is also proved for $\calf''$.\end{proof}

\begin{lemma}\label{trivalent}  Suppose $T\subset V$ is a compact, rooted tree with root vertex $v_T$, edge set $\cale$ and frontier $\calf$, where each vertex of $T$ has valence at least three in $V$.  There is a compact, rooted tree $T_0\subset V_0$ with root vertex $v_{T_0}$, edge set $\cale_0$ and frontier $\calf_0$, where each vertex of $T_0$ is trivalent in $V_0$, with the following property.

There is a bijection $q\co\calf\to\calf_0$ such that for any $\bd_{\calf}\in(\mathbb{R}^+)^{\calf}$, the tuple $\bd_{\calf_0}\in(\mathbb{R}^+)^{\calf_0}$ given by relabeling entries of $\bd_{\calf}$ using $q$ has the property that:
$$\min\{D_T(\bd)\,|\,\bd\in\overline{\mathit{Ad}}(\bd_{\calf})\} \geq \min\{D_{T_0}(\bd)\,|\,\bd\in\overline{\mathit{Ad}}(\bd_{\calf_0})\}$$\end{lemma}

\begin{proof}  Suppose $T$ has $k$ vertices.  Each vertex of $T$ has valence at least three in $V$, so $3k\leq 2|\cale|+|\calf|$, with equality holding if and only if each vertex of $T$ is trivalent in $V$.  Since $T$ is a tree its Euler characteristic is $1$, so we also have $|\cale| = k-1$.  Substituting this into the first inequality gives $k\leq |\calf| - 2$, with equality if and only if each vertex of $T$ is trivalent in $V$.  We will prove the Lemma by fixing $n = |\calf|\geq 3$ and inducting on $n - k$.  The base case $n-k = 2$ holds trivially with $T_0 = T$ and $q$ the identity map.

Let us now take $k < n-2$ and suppose the Lemma holds for all trees with frontier of order $n$ and more than $k$ vertices.  If $T\subset V$ is a compact, rooted tree with root vertex $v_T$, edge set $\cale$ of order $k$, and frontier $\calf$ of order $n$ such that each vertex of $T$ has valence at least three in $V$ then by the first paragraph there is a vertex $v$ of $T$ with valence at least four in $V$.  List the edges in $\cale\cup\calf$ containing $v$ as $e_0,\hdots,e_{n_v-1}$, where $n_v$ is the valence of $v$ in $V$, and if $v\neq v_T$ then $e_0 = e_v$ is the initial edge of the arc in $T$ joining $v$ to $v_T$.

Let $T_1$ be the maximal subtree of $T$ containing $e_0$ and $e_1$ but not $e_i$ for $i>1$, and let $T_2$ be the maximal subtree containing the remaining $e_i$ but not $e_0$ or $e_1$.  Then $T = T_1\cup T_2$ and $T_1\cap T_2 = \{v\}$.  We produce a tree $T'$ with $k+1$ vertices by joining a copy of $T_1$ to a copy of $T_2$ by an edge $e'$ that has its endpoints at the respective copies of $v$ in $T_1$ and $T_2$.  We produce $V'$ containing $T'$ similarly, by doubling $v$ and joining the resulting copies by $e'$.

There is a quotient map $p_{e'}\co V'\to V$ that identifies $e$ to a point and takes $T'$ to $T$.  It induces a bijection from $\cale'-\{e'\}$ to $\cale$, where $\cale'$ is the edge set of $T'$, and from the frontier $\calf'$ of $T'$ in $V'$ to $\calf$.  We will refer by $q'$ to refer to the inverse bijections both $\calf\to\calf'$ and $\cale\to\cale'-\{e'\}$.  Given $\bd_{\calf} = (d_f\,|\,f\in\calf)\in(\mathbb{R}^+)^{\calf}$ one produces $\bd'_{\calf'}\in(\mathbb{R}^+)^{\calf}$ by relabeling: $\bd'_{\calf'} = (d_{q(f)}\,|\,f\in\calf)$.  Similarly, given $\bd_{\cale}\in(\mathbb{R}^+)^{\cale}$, relabeling gives all entries of an element $\bd'_{\cale'}\in(\mathbb{R}^+)^{\cale'}$ but one, $d_{e'}$.  If $\bd = (\bd_{\cale},\bd_{\calf})\in\overline{\mathit{Ad}}(\bd_{\calf})$, we will choose $d_{e'}$ and a root vertex $v_{T'}$ for $T'$ so that the resulting element $\bd' = (\bd'_{\cale'},\bd'_{\calf'})$ lies in $\overline{\mathit{Ad}}(\bd'_{\calf'})$.

With the edges of $T$ containing $v$ enumerated as above let $d_i = d_{e_i}$ for each $i<n_v$, and define $d_{e'} = \ell_{n_v-1,1}(d_0,\hdots,d_{n_v-1})$, where $\ell_{i,j}$ is the diagonal-length function described in Corollary 1.15 of \cite{DeB_cyclic_geom}.  By that result, $d_{e'}$ is the length of the diagonal $\gamma$ of a cyclic $n_v$-gon $C_v$ with side length collection $(d_0,\hdots,d_{n_v-1})$ that cuts off the sides with lengths $d_0$ and $d_1$ from the others.  We now take $\bd'_{\cale'}$ as suggested in the previous paragraph, with $d_{q(e)} = d_e$ for each in $\cale$ and $d_{e'}$ as given here.

Now we assign $T'$ a root vertex $v_{T'}$.  Let $v_1$ be the copy of $v$ that lies in $T_1\subset T'$, and let $v_2$ be the other copy of $v$ in $T'$.  If $v\neq v_T$ we let $v_{T'} = p_{e'}^{-1}(v_T)$, a vertex of $T'$ since $p_{e'}$ is injective away from $e'$.  Now suppose $v = v_T$.  If the circumcircle center of $C_v$ lies on the side of $\gamma$ containing the edges of length $d_0$ and $d_1$ then we let $v_1 = v_{T'}$; otherwise we let $v_2 = v_{T'}$.

For $\bd' = (\bd'_{\cale'},\bd'_{\calf'})$ as prescribed above, we claim that $\bd'\in\overline{\mathit{Ad}}(\bd'_{\calf'})$.  In the case that $v\neq v_T$, condition (2) of Definition \ref{admissible closure} follows immediately by construction.  If $v = v_T$ then since $P_v(\bd)\in\calc_{n_T}\cup\calBC_{n_T}$, Proposition 2.2 of \cite{DeB_cyclic_geom} implies that the cyclic $n_v$-gon $C_v$ described above contains its circumcircle center $c$.  The diagonal $\gamma$ above divides $C_v$ into cyclic $n$-gons $C_{v_1}$ and $C_{v_2}$ with respective side length collections $P_{v_1}(\bd')$ and $P_{v_2}(\bd')$, where $P_{v_1}(\bd') = (d_0,d_1,d_{e'})$ and $P_{v_2}(\bd') = (d_{e'},d_2,\hdots,d_{n_v-1})$.  We chose to label $v_1$ or $v_2$ as $v_{T'}$ according to which of $C_{v_1}$ or $C_{v_2}$ contains $c$, so again by \cite[Prop.~2.2]{DeB_cyclic_geom} we have $P_{v_{T'}}(\bd')\in\calc_{n_{T'}}\cup\calBC_{n_{T'}}$.

By construction $J(P_{v_1}(\bd')) = J(P_{v_2}(\bd')) = J(P_v(\bd))$, and condition (3) of Definition \ref{admissible closure} follows.  Condition (1) of Definition \ref{admissible closure} follows for any vertex $v'$ of $T'$ outside $e'$ from the fact that $P_{v'}(\bd') = P_{p_{e'}(v')}(\bd)$.  We now check it for $v_1$ and $v_2$.

If $v\neq v_T$ then $P_v(\bd) = (d_0,\hdots,d_{n_v-1})\in\calAC_{n_v}-\calc_{n_v}$ has largest entry $d_0$, since we enumerated the edges containing $v$ as $e_0,\hdots,e_{n_v-1}$ so that $e_0 = e_v$.  Then by \cite[Prop. 2.2]{DeB_cyclic_geom} the side of $C_v$ with length $d_0$ separates $C_v$ from its circumcircle center $c$.  It therefore also separates $C_{v_1}$ from $c$, and $\gamma$ separates $C_{v_2}$ from $c$, which is their shared circumcircle center.  Thus again by \cite[Prop.~2.2]{DeB_cyclic_geom}, $P_{v_1}(\bd')\in\calAC_3-\calc_3$ has largest entry $d_0$ and $P_{v_2}(\bd')\in\calAC_{n_v-2}-\calc_{n_v-2}$ has largest entry $d_{e'}$.  Since $e_v$ is the initial edge of the arc joining $v$ to $v_T$ in $T$, this arc lies in $T_1$ and joins $v_1$ to $v_{T'} = v_T$ in $T_1\subset T'$.  Its initial edge is still $e_v$, and hence $e'$ is the initial edge of the arc in $T'$ joining $v_2$ to $v_{T'}$.

If $v = v_T$ and $v_2 = v_{T'}$ then $\gamma$ separates $C_{v_1}$ from $c$, so again Proposition 2.2 of \cite{DeB_cyclic_geom} implies that $d_{e'}$ is the largest entry of $P_{v_1}(\bd')\in\calAC_3-\calc_3$.  Since $e'$ is the arc joining $v_1$ to $v_2 = v_{T'}$, Definition \ref{admissible closure}(1) follows in this case.  The case that $v_1 = v_{T'}$ is completely analogous, and we have proven the claim that $\bd'\in\overline{\mathit{Ad}}(\bd'_{\calf'})$.

Our construction of $T'$ and $\bd'$ has reverse-engineered the hypotheses of Lemma 3.28 of \cite{DeB_Voronoi}, since $J(P_{v_1}(\bd')) = J(P_{v_2}(\bd'))$ and $T$ is obtained from $T'$ by crushing $e'$ to a point.  In the notation of that result and Definition 3.26 there, $T = T'_{e'}$ and $\bd = \bd'_{e'}$.  Therefore that Lemma gives $D_T(\bd) = D_{T'}(\bd')$.  Choosing $\bd$ as a minimizer for $D_T$ over $\overline{\mathit{Ad}}(\bd_{\calf})$ gives:
$$\min\{D_T(\bd)\,|\,\bd\in\overline{\mathit{Ad}}(\bd_{\calf})\} \geq \min\{D_{T'}(\bd')\,|\,\bd'\in\overline{\mathit{Ad}}(\bd'_{\calf'})\}$$
Now applying the induction hypothesis to $T'$, which has one more vertex than $T$, we conclude that the Lemma holds for $T$.  Thus by induction the Lemma holds for all trees with frontier of order $n\geq 3$.  But $n$ is arbitrary, so the Lemma holds.\end{proof}

\begin{theorem}\label{upgrade}\Upgrade\end{theorem}

\begin{proof}Let $T\subset V$ be the dual tree to $C$ (recall Definition 2.11 of \cite{DeB_Voronoi}), where $V$ is the Voronoi tessellation's one-skeleton, and enumerate the frontier $\calf$ of $T$ as $\{f_1,\hdots,f_n\}$ so that the edge of $C$ dual to $f_i$ has length at least $b_i$ for each $i$.  Let $d_i$ be the length of this edge, and let $\bd_{\calf} = (d_1,\hdots,d_n)\in(\mathbb{R}^+)^n$.  Taking $\bd_{\cale}$ to be the tuple of lengths of Delaunay edges dual to edges of $T$, Lemma 3.14 of \cite{DeB_Voronoi} implies that $\bd = (\bd_{\cale},\bd_{\calf})$ lies in $\mathit{Ad}(\bd_{\calf})\subset\overline{\mathit{Ad}}(\bd_{\calf})$, and the area of $C$ is $D_T(\bd)$.  

By Lemma \ref{trivalent} there is a tree $T_0\subset V_0$ with frontier $\calf_0$ bijective to $\calf$, such that each vertex of $T_0$ is trivalent in $V_0$, with the property that relabeling the entries of $\bd_{\calf}$ using the bijection $\calf\to\calf_0$ yields $\bd_{\calf_0}$ satisfying: 
$$D_T(\bd) \geq \min\{D_{T_0}(\bd_0)\,|\,\bd_0\in\overline{\mathit{Ad}}(\bd_{\calf_0})\}$$  
By Proposition \ref{centered lower}, this quantity in turn is bounded below by $B_{T_0}(\bfb_{\calf_0})$, where $\bfb_{\calf_0}$ is obtained from $\bfb_{\calf}$ by relabeling in the same way.  The Theorem follows.\end{proof}

\begin{remark}\label{coset}  The area function $D_0$ is symmetric in its inputs \cite[Prop.~2.3]{DeB_cyclic_geom}, and the semicyclic radius function $b_0$ is too \cite[Prop.~1.12]{DeB_cyclic_geom}.  Using these facts it is not hard to show that for any $\bfb$ and tree $T$, if edges $f_1$ and $f_2$ of $\calf$ terminate at the same vertex of $T$ then for the transposition $\tau$ that swaps the corresponding entries $b_{f_1}$ and $b_{f_2}$ of $\bfb$, $b_e(\bfb) = b_e(\tau(\bfb))$ for each $e\in\cale$ and $B_T(\bfb) = B_T(\tau(\bfb))$.  So in computing $\min\{B_T(\bfb)\}$ above, for each tree $T$ it is only necessary to test one representative of each left coset of the subgroup $S_T$ (isomorphic to a direct sum of $\mathbb{Z}_2$'s) of $S_n$ generated by such swaps.

Furthermore, an automorphism $f$ of $(T,v_T)$ has an induced action on $\bfb$ which is well-defined up to the action of $S_T$, where the edges of $\calf$ that terminate at $v$ are taken to those that terminate at $f(v)$ for each $v\in T^{(0)}$ and the corresponding entries of $\bfb$ go along for the ride.  One can show again that $b_e(\bfb) = b_e(f(\bfb))$ for each $e\in\cale$, and $B_T(\bfb) = B_T(f(\bfb))$.  Thus for each tree $T$ it is in fact only necessary to test $B_T(\sigma(\bfb))$ for representatives $\sigma$ of each orbit of the action of the automorphism group of $(T,v_T)$ on $S_n/S_T$.\end{remark}

\section{Practice}\label{practice}

The main goal of this section is to describe the Python module \textit{minimizer.py} and data file \textit{forest.txt}, which together give us the ability to obtain the bounds of Theorem \ref{upgrade} for arbitrary $n$-tuples, $n\leq 9$, using a computer.  First, in subsection \ref{formulas} we record some existing explicit formulas for geometric measurements of cyclic polygons.  We use these to give a completely explicit statement of Theorem \ref{upgrade}, in Corollary \ref{better upgrade}.  Then in subsection \ref{programs} we describe \textit{forest.txt} and the components of \textit{minimizer.py}.

Here is a brief explanation of how to use \textit{minimizer} to compute the bounds of Theorem \ref{upgrade}.  After downloading \textit{minimizer} and \textit{forest} you must first replace ``\texttt{yourpath}'' on line 66 of \textit{minimizer.py}, with your path to \textit{forest.txt}.  Then in a Python interpreter, import \textit{minimizer.py} and run \textit{minimizer.minimize()} on the desired tuples.  Here is a sample series of commands at the Python command prompt, to get it up and running:
\begin{quote} \texttt{$>>>$ import sys\\$>>>$ sys.path.append(`}[Your path to \textit{minimizer.py}]\texttt{')\\$>>>$ import minimizer\\$>>>$ minimizer.minimize([1,2,3,4,5])} \end{quote}
This computes Theorem \ref{upgrade}'s lower bound on the area of a five-edged centered dual two-cell with edge lengths bounded below by $(b_1,b_2,b_3,b_4,b_5)$, where $\sinh(b_i/2) = i$ for each $i$ (see Important Note \ref{heythere}).

Finally, in subsection \ref{examples} we prove Proposition \ref{bigger than}, on the relationship between Theorem \ref{upgrade} and Theorem 3.31 of \cite{DeB_Voronoi}, and consider some illustrative examples.

\subsection{Formulas}\label{formulas}  The reduction to the trivalent case allowed by Lemma \ref{trivalent} gives a huge savings in computational expense, since there are explicit formulas for two critical quantities: the triangle area, and the circumcircle radius of a semicyclic triangle.  For an arbitrary cyclic $n$-gon $C$ we do not know an explicit formula in terms of side length for the area of $C$, and the same is true for semicyclic circumcircle radius.  Below we cite the references we know for the results we use.  Any omissions are due to the author's ignorance, and additional references are welcome.

The ``hyperbolic Heron formula'' below was first recorded (to my knowledge) by S.~Bilinski \cite{Bilinski}.  An equivalent formulation was rediscovered by W.W.~Stothers \cite{Stothers}.

\begin{lemma}\label{heron}The area of a compact hyperbolic triangle with sides of length $a$, $b$ and $c$ is:
$$D_0(a,b,c) = 2\cos^{-1}\left(\frac{\sinh^2(a/2)+\sinh^2(b/2)+\sinh^2(c/2) + 2}{2\cosh(a/2)\cosh(b/2)\cosh(c/2)}\right) $$
Here $D_0$ refers to the area function from \cite[Prop.~2.3]{DeB_cyclic_geom}.\end{lemma}

Recall that we say a cyclic triangle is \textit{semicyclic} if its longest side is also a diameter of its circumcircle or, equivalently, if its side length collection $(a,b,c)$ lies in the space $\calBC_3$ of \cite[Prop.~1.11]{DeB_cyclic_geom}.  The ``Pythagorean theorem for semicyclic hyperbolic triangles'' below is recorded as Lemma 4.3 of N\"a\"at\"anen--Penner \cite{NaatPen}.  It can easily be derived from the hyperbolic law of sines (see eg.~\cite[\S 7.12]{Beardon}).

\begin{lemma}\label{pythagoras}The circumcircle radius $J$ of a compact, semicyclic hyperbolic triangle with shorter side lengths $a$ and $b$ satisfies:
$$\sinh^2 J = \sinh^2(a/2)+\sinh^2(b/2)$$
Equivalently $b_0(a,b) = 2\sinh^{-1}\left(\sqrt{\sinh^2(a/2)+\sinh^2(b/2)}\right)$, for $b_0$ as in \cite[Prop.~1.12]{DeB_cyclic_geom}.\end{lemma}

\begin{corollary}\label{semicyclic tri}The area of a compact, semicyclic hyperbolic triangle with shorter side lengths $a$ and $b$ is
$$ D_0(a,b,b_0(a,b)) = 2\sin^{-1}\left(\frac{\sinh(a/2)\sinh(b/2)}{\cosh(a/2)\cosh(b/2)}\right) = 2\sin^{-1}\left(\tanh(a/2)\tanh(b/2)\right)$$
\end{corollary}

\begin{proof}  This follows by simply substituting the formula for $b_0$ from Lemma \ref{pythagoras} for $c$ in the formula for $D_0$ from Lemma \ref{heron}.  Letting ``$D_0$'' refer to $D_0(a,b,b_0(a,b))$ and ``$A$'' and ``$B$'' to $\sinh(a/2)$ and $\sinh(b/2)$, respectively, we have:\begin{align}\label{pythtoheron}
    \cos(D_0/2) = \frac{2(A^2+B^2+1)}{2\sqrt{A^2+B^2+1}\cosh(a/2)\cosh(b/2)} = \frac{\sqrt{A^2+B^2+1}}{\cosh(a/2)\cosh(b/2)} \end{align}
Applying the identity $\sin^2\theta = 1-\cos^2\theta$ and simplifying gives the result.\end{proof}

We now use the formulas above to give a self-contained, explicit statement of Theorem \ref{upgrade}.

\begin{corollary}\label{better upgrade}Let $C$ be a compact two-cell of the centered dual complex of a locally finite set $\cals\subset\mathbb{H}^2$ such that for some $\bfb = (b_1,\hdots,b_n)\in(\mathbb{R}^+)^n$ and enumeration of the edges of $C$, the $i^{\mathit{th}}$ edge has length at least $b_i$ for each $i$.  Then $\mathrm{area}(C)\geq\min\{B_T(\bfb)\,|\,T\in\calt_n,\sigma\in S_n\}$, where $S_n$ is the symmetric group on $n$ letters, $\sigma\in S_n$ acts on $\bfb$ by permutation of entries, and $\calt_n$ is the collection of all compact, rooted trees $T$ with root vertex $v_T$, frontier $\calf$ of order $n$, and each vertex trivalent in $T\cup\bigcup_{f\in\calf} f$; and:
$$ B_T(\bfb) = 2\cos^{-1}\left(\frac{2+\sum_{i=0}^2\sinh^2(m_{e_T^i}/2)}{2\prod_{i=0}^2\cosh(m_{e_T^i}/2)}\right) + \sum_{v\in T^{(0)}-\{v_T\}} 2\sin^{-1}\left(\tanh(b_{e_v^1}/2)\tanh(b_{e_v^2}/2)\right) $$
Here for $v\in T^{(0)}-\{v_T\}$, $e_v^1$ and $e_v^2$ are the two edges containing $v$ with the property that $v$ is closer to $v_T$ than the other endpoint.  For each edge $e$ of $T$, taking $v$ to be the further endpoint of $e$ from $v_T$, we recursively define $b_e=b_e(\bfb)$ following \cite[Lemma 3.19]{DeB_Voronoi}:
$$ b_e(\bfb) = 2\sinh^{-1}\sqrt{\sinh^2(b_{e_v^1}/2)+\sinh^2(b_{e_v^2}/2)}, $$
The three edges containing $v_T$ are enumerated as $e_T^0$, $e_T^1$, and $e_T^2$, and for each $i$, taking $i\pm 1$ modulo three, we define:
$$m_{e_T^i} = \min\left\{b_{e_T^i},2\sinh^{-1}\sqrt{\sinh^2(b_{e_T^{i+1}}/2)+\sinh^2(b_{e_T^{i-1}}/2)}\right\},$$
In particular, $m_{e_T^i} = b_{e_T^i}$ if $b_{e_T^i}$ is not maximal.\end{corollary}

\begin{proof}  This is obtained from Theorem \ref{upgrade} by writing out the formula for $B_T(\bfb)$ from Proposition \ref{centered lower} and noting that since each vertex of $T$ has valence three in $T\cup\bigcup_{f\in\calf} f$, Lemma \ref{heron} computes $D_0(m_{e_0},b_{e_1},\hdots,b_{e_{n_T}-1})$, Lemma \ref{pythagoras} computes $b_e(\bfb)$ for each $e\in\cale$, and Corollary \ref{semicyclic tri} computes $D_0(P_v(\bfb_{\cale},\bfb))$ for all $v\in T^{(0)}-\{v_T\}$.\end{proof}

\subsection{Programs}\label{programs}  This section describes \textit{minimizer.py}, a Python module containing a script \textit{minimize()} for computing the lower bounds given by Corollary \ref{better upgrade} on areas of centered dual two-cells with at most nine edges.  The architecture of \textit{minimize()} is simple, and we write it here in pseudocode:

\begin{quote}\texttt{define minimize(bounds)\\ \indent n = length(bounds), minlb = -1 \\ \indent for tree in forest(n)\\ \indent\indent for b in permute(bounds)\\ \indent\indent\indent lb = treecrawler(tree,b)\\ \indent\indent\indent if minlb == -1 then minlb = lb\\ \indent\indent\indent else minlb = min(lb, minlb) \\ return minlb}\end{quote}

Given an input $n$-tuple ``bounds'', the idea is to loop over each tree in $\calt_n$, and for each tree $T$ over each permutation $\bfb$ of bounds, computing $B_T(\bfb)$ and comparing it to the minimum obtained from prior computations.  Here \textit{forest()} is a routine that produces all elements of $\calt_n$ for a given $n$; \textit{permute()} produces all permutations of a given tuple; and \textit{treecrawler(,)} computes $B_T(\bfb)$, given $T\in\calt_n$ and an $n$-tuple $\bfb$.

\begin{important}\label{heythere}  For a given $n$-tuple $\bfb = (b_1,\hdots,b_n)$, to obtain $\min\{B_T(\sigma(\bfb))\}$ one inputs a list $[B_1,\hdots,B_n]$ to \textit{minimize()}, where $B_i = \sinh(b_i/2)$ for each $i$.  The motivation for this choice is the nature of the explicit functions in Section \ref{formulas}.\end{important}

Below we give some details on the implementations of \textit{permute}, \textit{treecrawler}, and \textit{forest}.

\subsubsection{Permute}  The \textit{itertools} Python module contains a script \textit{permutations} that generates all permutations of a given list.  Our implementation of \textit{minimizer.py} calls this function to produce permutations of $\bfb$.  For custom applications or use in programming languages that lack such a pre-built tool, we note that many existing permutation generation algorithms can be found, eg.~in \cite{Knuth} or on Wikipedia.  We point out one concrete example: the ``Steinhaus--Johnson--Trotter algorithm'', which was later improved by S.~Even, see \cite{Even}.

Another thing to note is that generating all permutations of any tuple of bounds produces considerable redundancy in the output of \textit{treecrawler(,)}, on account of Remark \ref{coset}.  Guided by the KISS principle (and our limitations as a coder), we have elected not to attempt to remove this redundancy in our implementation.

\subsubsection{Treecrawler}  This function from \textit{minimizer.py} takes two lists as input: one, ``tree'', of length $n-3$ which encodes a rooted tree $T$ with frontier $\calf$ of order $n$, and another, ``bound'', of length $n$ which contains an entry $B_i = \sinh(b_i/2)$ for each entry $b_i$ of a tuple $\bfb$ of edge length bounds.  A couple of observations motivate choosing this form for the inputs.  First,

\begin{para}\label{numbers}Every tree in $\calt_n$ has $n-3$ edges.\end{para}

Let $\cale$ be the edge set of $T$.  Since each edge in $\cale$ contains two vertices of $T$ and each edge of $\calf$ contains exactly one, by our trivalence hypothesis the number $k$ of vertices of $T$ satisfies $3k = 2|\cale|+n$.  Since $T$ is a tree its Euler characteristic is $k - |\cale| = 1$, so substituting gives $k = n-2$ and $|\cale| = n-3$.

The second motivating observation is:

\begin{para}\label{figures}  The vertices of a compact, rooted tree $T$ with $k$ edges and root vertex $v_T$ can be enumerated as $v_0,\hdots,v_k$ so that for each $i$ and $j$, if the arc $[v_i,v_T]$ from $v_i$ to $v_T$ contains $v_j$ then $i\leq j$.  Given such a numbering, enumerate the edges of $T$ as $e_0,\hdots,e_{k-1}$ so that $e_i$ is the initial edge of $[v_i,v_T]$ for each $i<k$.  Then $T$ is determined by the $k$-tuple $(n_0,\hdots,n_{k-1})$, where for each $i$, $v_{n_i}$ is the nearer vertex of $e_i$ to $v_T$.\end{para}

One may produce the desired enumeration of the vertices of $T$ by first listing all those at maximal distance $d$ from $v_T$ in $T$, then listing those at distance $d-1$, and so forth.  Note that any such enumeration has $v_T = v_k$.  And since $T$ is a tree there is a \textit{unique} arc joining $v_i$ to $v_T$ for all $i<k$, so $e_i$ is uniquely determined for each such $i$ by the requirement above.  This yields $k$ unique edges; all of them, since $T$ is a tree with $k+1$ vertices.

Figure \ref{tree strings} depicts the rooted trees with one to three edges, with vertices enumerated and the resulting encoding tuples following \ref{figures}.

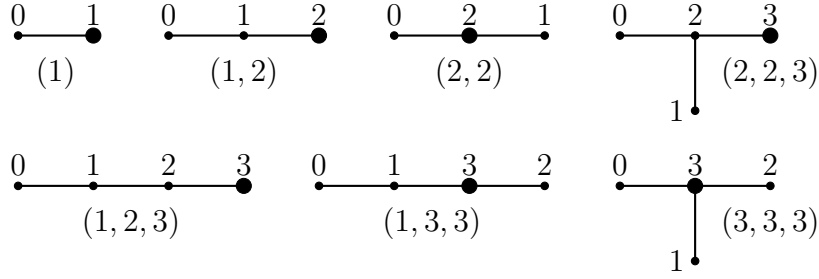
\begin{figure}
\begin{tikzpicture}

\draw [thick] (0,0) -- (1,0);
\draw [fill] (0,0) circle [radius=0.05];
\draw [fill] (1,0) circle [radius=0.1];
\node [above] at (0,0) {$0$};
\node [above] at (1,0) {$1$};
\node at (0.5,-0.5) {$(1)$};

\draw [thick] (2,0) -- (4,0);
\draw [fill] (2,0) circle [radius=0.05];
\draw [fill] (3,0) circle [radius=0.05];
\draw [fill] (4,0) circle [radius=0.1];
\node [above] at (2,0) {$0$};
\node [above] at (3,0) {$1$};
\node [above] at (4,0) {$2$};
\node at (3,-0.5) {$(1,2)$};

\draw [thick] (5,0) -- (7,0);
\draw [fill] (5,0) circle [radius=0.05];
\draw [fill] (6,0) circle [radius=0.1];
\draw [fill] (7,0) circle [radius=0.05];
\node [above] at (5,0) {$0$};
\node [above] at (6,0) {$2$};
\node [above] at (7,0) {$1$};
\node at (6,-0.5) {$(2,2)$};

\draw [thick] (8,0) -- (10,0);
\draw [fill] (8,0) circle [radius=0.05];
\draw [fill] (9,0) circle [radius=0.05];
\draw [fill] (10,0) circle [radius=0.1];
\draw [thick] (9,0) -- (9,-1);
\draw [fill] (9,-1) circle [radius=0.05];
\node [above] at (8,0) {$0$};
\node [above] at (9,0) {$2$};
\node [above] at (10,0) {$3$};
\node [left] at (9,-1) {$1$};
\node at (10,-0.5) {$(2,2,3)$};

\draw [thick] (8,-2) -- (10,-2);
\draw [fill] (8,-2) circle [radius=0.05];
\draw [fill] (9,-2) circle [radius=0.1];
\draw [fill] (10,-2) circle [radius=0.05];
\draw [thick] (9,-2) -- (9,-3);
\draw [fill] (9,-3) circle [radius=0.05];
\node [above] at (8,-2) {$0$};
\node [above] at (9,-2) {$3$};
\node [above] at (10,-2) {$2$};
\node [left] at (9,-3) {$1$};
\node at (10,-2.5) {$(3,3,3)$};

\draw [thick] (0,-2) -- (3,-2);
\draw [fill] (0,-2) circle [radius=0.05];
\draw [fill] (1,-2) circle [radius=0.05];
\draw [fill] (2,-2) circle [radius=0.05];
\draw [fill] (3,-2) circle [radius=0.1];
\node [above] at (0,-2) {$0$};
\node [above] at (1,-2) {$1$};
\node [above] at (2,-2) {$2$};
\node [above] at (3,-2) {$3$};
\node at (1.5,-2.5) {$(1,2,3)$};

\draw [thick] (4,-2) -- (7,-2);
\draw [fill] (4,-2) circle [radius=0.05];
\draw [fill] (5,-2) circle [radius=0.05];
\draw [fill] (6,-2) circle [radius=0.1];
\draw [fill] (7,-2) circle [radius=0.05];
\node [above] at (4,-2) {$0$};
\node [above] at (5,-2) {$1$};
\node [above] at (6,-2) {$3$};
\node [above] at (7,-2) {$2$};
\node at (5.5,-2.5) {$(1,3,3)$};

\end{tikzpicture}
\caption{The rooted trees with $2\leq k\leq 4$ vertices, with root vertices enlarged and vertices labeled, and the associated $(k-1)$-tuples.}
\label{tree strings}
\end{figure}

Treecrawler distributes the bounds of ``bound'' to frontier edges of $T$ in the opposite order from that of the vertices.  That is, having enumerated the vertices of $T$ as $v_0,\hdots,v_{n-3}$ following \ref{figures}, we enumerate $\calf$ as $\{f_1,\hdots,f_n\}$ so that if $i < j$ then $n_i\geq n_j$, where $v_{n_i}$ and $v_{n_j}$ are the vertices of $T$ contained $f_i$ and $f_j$, respectively.  Then we bound $d_{f_i}$ below by $b_i$ for each $i$.  (This choice is made because it is computationally least expensive.)

The idea of the program is to recursively compute $b_{e_i}$ and $D_0(P_{v_i}(\bfb))$ using the formulas of Lemmas \ref{heron} and \ref{pythagoras}.  Having done so for a given $i$, we append $b_{e_i}$ to a list ``edgelengths'' containing the lengths $e_j$, and we add $D_0(P_{v_i}(\bfb))$ to the number ``totalarea'' that records the sum of the $D_0(P_{v_j}(\bfb))$, for $j< i$.  The main observation here is that for each $i< n-3$, $v_i$ is contained in exactly two edges of $\cale\cup\calf$ not equal to $e_i$, and if either of these is of the form $e_j\in\cale$ then $j<i$.

\subsubsection{Forest}  This is a library of all elements of $\calt_n$ for $3\leq n\leq 9$.  Recall that an element $T$ of $\calt_n$ is a compact, rooted tree with frontier $\calf$ of order $n$, such that each vertex of $T$ has valence three in $T\cup\bigcup_{f\in\calf} f$, and therefore valence at most three in $T$.  In \textit{forest.txt} and Figure \ref{rooted trees}, which depicts its members, we track only the internal structure of $T$ --- the idea being that each vertex of $T$ gets as many frontier edges as necessary to bring its valence up to three in $T\cup\bigcup_{f\in\calf} f$.

By \ref{numbers} above, each $T\in\calt_n$ has $k =n-3$ edges, so \textit{forest.txt} encodes $T$ by a string of length $k$ following the scheme in \ref{figures}.  The first line of \textit{forest.txt} is a key: a ($0$-based) list $L$ whose $k$th entry $L[k]$ records the number of lines down that the codes for trees with $k$ edges begin.  In fact the codes begin one line below that: the line $L[k]$ below the first contains a single integer, which is the total number of codes for $k$-edged trees.  Each tree code occupies a single line.  It is the code described in \ref{figures} but written in reverse order (\textit{minimizer.py} un-reverses the order when reading the code).

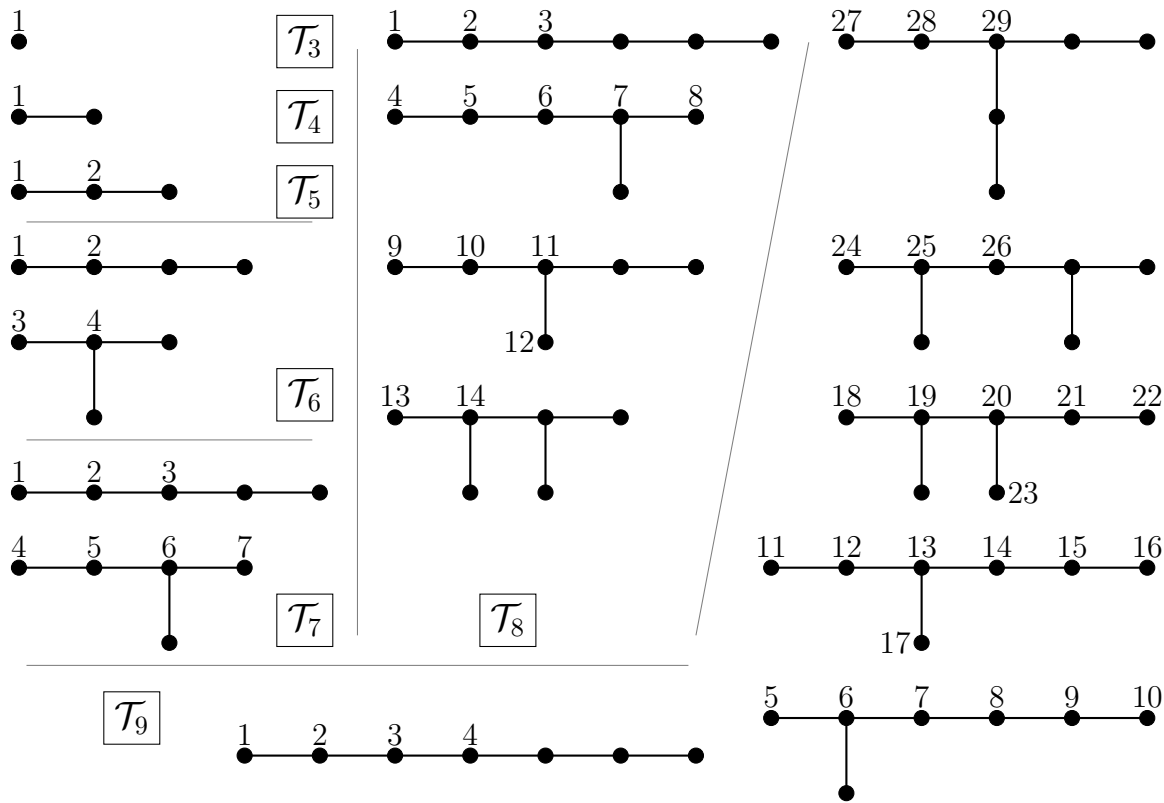
\begin{figure}
\begin{tikzpicture}

\node [rectangle, draw=black] at (3.8,0) {\large $\calt_3$};
\draw [fill] (0,0) circle [radius=0.1];
\node [above] at (0,0) {$1$};

\node [rectangle, draw=black] at (3.8,-1) {\large$\calt_4$};
\draw [thick] (0,-1) -- (1,-1);
\draw [fill] (0,-1) circle [radius=0.1];
\draw [fill] (1,-1) circle [radius=0.1];
\node [above] at (0,-1) {$1$};

\node [rectangle, draw=black] at (3.8,-2) {\large$\calt_5$};
\draw [thick] (0,-2) -- (2,-2);
\draw [fill] (0,-2) circle [radius=0.1];
\draw [fill] (1,-2) circle [radius=0.1];
\draw [fill] (2,-2) circle [radius=0.1];
\node [above] at (1,-2) {$2$};
\node [above] at (0,-2) {$1$};

\draw [draw=gray, thin] (.1,-2.4) -- (3.9,-2.4);
\node [rectangle, draw=black] at (3.8,-4.7) {\large $\calt_6$};

\draw [thick] (0,-3) -- (3,-3);
\draw [fill] (0,-3) circle [radius=0.1];
\draw [fill] (1,-3) circle [radius=0.1];
\draw [fill] (2,-3) circle [radius=0.1];
\draw [fill] (3,-3) circle [radius=0.1];
\node [above] at (1,-3) {$2$};
\node [above] at (0,-3) {$1$};

\draw [thick] (0,-4) -- (2,-4);
\draw [thick] (1,-4) -- (1,-5);
\draw [fill] (0,-4) circle [radius=0.1];
\draw [fill] (1,-4) circle [radius=0.1];
\draw [fill] (2,-4) circle [radius=0.1];
\draw [fill] (1,-5) circle [radius=0.1];
\node [above] at (1,-4) {$4$};
\node [above] at (0,-4) {$3$};

\draw [draw=gray, thin] (.1,-5.3) -- (3.9,-5.3);
\node [rectangle, draw=black] at (3.8,-7.7) {\large $\calt_7$};

\draw [thick] (0,-6) -- (4,-6);
\draw [fill] (0,-6) circle [radius=0.1];
\draw [fill] (1,-6) circle [radius=0.1];
\draw [fill] (2,-6) circle [radius=0.1];
\draw [fill] (3,-6) circle [radius=0.1];
\draw [fill] (4,-6) circle [radius=0.1];
\node [above] at (2,-6) {$3$};
\node [above] at (1,-6) {$2$};
\node [above] at (0,-6) {$1$};

\draw [thick] (0,-7) -- (3,-7);
\draw [thick] (2,-7) -- (2,-8);
\draw [fill] (0,-7) circle [radius=0.1];
\draw [fill] (1,-7) circle [radius=0.1];
\draw [fill] (2,-7) circle [radius=0.1];
\draw [fill] (3,-7) circle [radius=0.1];
\draw [fill] (2,-8) circle [radius=0.1];
\node [above] at (0,-7) {$4$};
\node [above] at (1,-7) {$5$};
\node [above] at (2,-7) {$6$};
\node [above] at (3,-7) {$7$};

\draw [draw=gray, thin] (4.5,-.1) -- (4.5,-7.9);
\node [rectangle, draw=black] at (6.5,-7.7) {\large$\calt_8$};

\draw [thick] (5,0) -- (10,0);
\draw [fill] (5,0) circle [radius=0.1];
\draw [fill] (6,0) circle [radius=0.1];
\draw [fill] (7,0) circle [radius=0.1];
\draw [fill] (8,0) circle [radius=0.1];
\draw [fill] (9,0) circle [radius=0.1];
\draw [fill] (10,0) circle [radius=0.1];
\node [above] at (5,0) {$1$};
\node [above] at (6,0) {$2$};
\node [above] at (7,0) {$3$};

\draw [thick] (5,-1) -- (9,-1);
\draw [thick] (8,-1) -- (8,-2);
\draw [fill] (5,-1) circle [radius=0.1];
\draw [fill] (6,-1) circle [radius=0.1];
\draw [fill] (7,-1) circle [radius=0.1];
\draw [fill] (8,-1) circle [radius=0.1];
\draw [fill] (9,-1) circle [radius=0.1];
\draw [fill] (8,-2) circle [radius=0.1];
\node [above] at (5,-1) {$4$};
\node [above] at (6,-1) {$5$};
\node [above] at (7,-1) {$6$};
\node [above] at (8,-1) {$7$};
\node [above] at (9,-1) {$8$};

\draw [thick] (5,-3) -- (9,-3);
\draw [thick] (7,-3) -- (7,-4);
\draw [fill] (5,-3) circle [radius=0.1];
\draw [fill] (6,-3) circle [radius=0.1];
\draw [fill] (7,-3) circle [radius=0.1];
\draw [fill] (8,-3) circle [radius=0.1];
\draw [fill] (9,-3) circle [radius=0.1];
\draw [fill] (7,-4) circle [radius=0.1];
\node [above] at (5,-3) {$9$};
\node [above] at (6,-3) {$10$};
\node [above] at (7,-3) {$11$};
\node [left] at (7,-4) {$12$};

\draw [thick] (5,-5) -- (8,-5);
\draw [thick] (6,-5) -- (6,-6);
\draw [thick] (7,-5) -- (7,-6);
\draw [fill] (5,-5) circle [radius=0.1];
\draw [fill] (6,-5) circle [radius=0.1];
\draw [fill] (7,-5) circle [radius=0.1];
\draw [fill] (8,-5) circle [radius=0.1];
\draw [fill] (6,-6) circle [radius=0.1];
\draw [fill] (7,-6) circle [radius=0.1];
\node [above] at (5,-5) {$13$};
\node [above] at (6,-5) {$14$};

\draw [draw=gray, thin] (10.5,-.01) -- (9,-7.9);
\node [rectangle, draw=black] at (1.5,-9) {\large$\calt_9$};

\draw [thick] (11,0) -- (15,0);
\draw [thick] (13,0) -- (13,-2);
\draw [fill] (11,0) circle [radius=0.1];
\draw [fill] (12,0) circle [radius=0.1];
\draw [fill] (13,0) circle [radius=0.1];
\draw [fill] (14,0) circle [radius=0.1];
\draw [fill] (15,0) circle [radius=0.1];
\draw [fill] (13,-1) circle [radius=0.1];
\draw [fill] (13,-2) circle [radius=0.1];
\node [above] at (11,0) {$27$};
\node [above] at (12,0) {$28$};
\node [above] at (13,0) {$29$};

\draw [thick] (11,-3) -- (15,-3);
\draw [thick] (12,-3) -- (12,-4);
\draw [thick] (14,-3) -- (14,-4);
\draw [fill] (11,-3) circle [radius=0.1];
\draw [fill] (12,-3) circle [radius=0.1];
\draw [fill] (13,-3) circle [radius=0.1];
\draw [fill] (14,-3) circle [radius=0.1];
\draw [fill] (15,-3) circle [radius=0.1];
\draw [fill] (12,-4) circle [radius=0.1];
\draw [fill] (14,-4) circle [radius=0.1];
\node [above] at (11,-3) {$24$};
\node [above] at (12,-3) {$25$};
\node [above] at (13,-3) {$26$};

\draw [thick] (11,-5) -- (15,-5);
\draw [thick] (12,-5) -- (12,-6);
\draw [thick] (13,-5) -- (13,-6);
\draw [fill] (11,-5) circle [radius=0.1];
\draw [fill] (12,-5) circle [radius=0.1];
\draw [fill] (13,-5) circle [radius=0.1];
\draw [fill] (14,-5) circle [radius=0.1];
\draw [fill] (15,-5) circle [radius=0.1];
\draw [fill] (12,-6) circle [radius=0.1];
\draw [fill] (13,-6) circle [radius=0.1];
\node [above] at (11,-5) {$18$};
\node [above] at (12,-5) {$19$};
\node [above] at (13,-5) {$20$};
\node [above] at (14,-5) {$21$};
\node [above] at (15,-5) {$22$};
\node [right] at (13,-6) {$23$};

\draw [thick] (10,-7) -- (15,-7);
\draw [thick] (12,-7) -- (12,-8);
\draw [fill] (10,-7) circle [radius=0.1];
\draw [fill] (11,-7) circle [radius=0.1];
\draw [fill] (12,-7) circle [radius=0.1];
\draw [fill] (13,-7) circle [radius=0.1];
\draw [fill] (14,-7) circle [radius=0.1];
\draw [fill] (15,-7) circle [radius=0.1];
\draw [fill] (12,-8) circle [radius=0.1];
\node [above] at (10,-7) {$11$};
\node [above] at (11,-7) {$12$};
\node [above] at (12,-7) {$13$};
\node [above] at (13,-7) {$14$};
\node [above] at (14,-7) {$15$};
\node [above] at (15,-7) {$16$};
\node [left] at (12,-8) {$17$};

\draw [thick] (10,-9) -- (15,-9);
\draw [thick] (11,-10) -- (11,-9);
\draw [fill] (10,-9) circle [radius=0.1];
\draw [fill] (11,-9) circle [radius=0.1];
\draw [fill] (12,-9) circle [radius=0.1];
\draw [fill] (13,-9) circle [radius=0.1];
\draw [fill] (14,-9) circle [radius=0.1];
\draw [fill] (15,-9) circle [radius=0.1];
\draw [fill] (11,-10) circle [radius=0.1];
\node [above] at (10,-9) {$5$};
\node [above] at (11,-9) {$6$};
\node [above] at (12,-9) {$7$};
\node [above] at (13,-9) {$8$};
\node [above] at (14,-9) {$9$};
\node [above] at (15,-9) {$10$};

\draw [thick] (3,-9.5) -- (9,-9.5);
\draw [fill] (3,-9.5) circle [radius=0.1];
\draw [fill] (4,-9.5) circle [radius=0.1];
\draw [fill] (5,-9.5) circle [radius=0.1];
\draw [fill] (6,-9.5) circle [radius=0.1];
\draw [fill] (7,-9.5) circle [radius=0.1];
\draw [fill] (8,-9.5) circle [radius=0.1];
\draw [fill] (9,-9.5) circle [radius=0.1];
\node [above] at (3,-9.5) {$1$};
\node [above] at (4,-9.5) {$2$};
\node [above] at (5,-9.5) {$3$};
\node [above] at (6,-9.5) {$4$};

\draw [draw=gray, thin] (0.1,-8.3) -- (8.9,-8.3);

\end{tikzpicture}
\caption{The collections $\calt_n$ for $n\leq 9$.}
\label{rooted trees}
\end{figure}

Figure \ref{rooted trees} depicts the $T\in\calt_n$ for $3\leq n\leq 9$.  For each $n$, trees are numbered in the figure according to their position in the corresponding list in \textit{forest.txt}: if a vertex $v$ of a tree $T$ with $n-3$ edges there is numbered $i$ then the element of $\calt_n$ represented by $T$ with root vertex $v$ is also represented by the $i$th code from the top of the list in \textit{forest.txt} containing codes for $\calt_n$.  For instance, for the vertex labeled $14$ on the five-edged tree $T$ in the figure, enumerating the vertices of $T$ following \ref{figures} yields the five-tuple $(2,2,5,5,5)$.  The reverse of this is the fourteenth and final 5-edged tree code in \textit{forest.txt}.  

A vertex is not numbered if and only if it is equivalent to one that is under a non-trivial automorphism of the tree that contains it.  Note that the vertex numbering prescribed in \ref{figures} is not necessarily unique, so each code is one of possibly several representing the same tree.  In the case above the other possibilities are $(3,3,5,5,5)$ and $(4,4,5,5,5)$.

We record a couple of observations to support the classification of trees in Figure \ref{rooted trees}.  A tree $T$ with $k$ edges has diameter $d\leq k$.  Let $\gamma$ be an edge arc of length equal to $d$ (such an arc is the horizontal part of each tree in the Figure).  The $k-d$ edges not contained in $\gamma$ lie in a disjoint union of maximal subtrees of $T$ that do not contain edges of $\gamma$.  If $T_0$ is such a subtree then its intersection point $v_0$ with $\gamma$ can be at most $d-d_0$ away from each boundary vertex of $\gamma$, where $d_0$ is the maximal distance in $T_0$ from $v_0$ to another vertex.  Otherwise the diameter hypothesis would be violated.

It is now just a matter of enumerating the possible $d\leq k$, and for each $d$, the possibilities for $T_0$.  We note also that $d$ cannot be too small.  For instance, in the six-edge case if $d$ were equal to $3$ then three edges of $T$ would lie outside $\gamma$.  There are only two possible attachment points for the trees $T_0$, the interior vertices of $\gamma$.  So one such $T_0$ must contain at least two edges, hence it must have $d_0\geq 2$.  But this contradicts our assumption on $d$, since each interior vertex of $\gamma$ is at distance $2$ from one of its boundary vertices.

\subsection{Examples}\label{examples}

Here we analyze a couple of examples that illustrate important basic features of the bounds of Theorem \ref{upgrade}.  Before the first, we pause to observe:

\begin{proposition}\label{bigger than}  For any $n>4$ and $d >0$ the bound $\min\{B_T(\sigma(\bfb))\}$ given by Theorem \ref{upgrade} for $\bfb = (d,\hdots,d)\in\mathbb{R}^n$ is larger than the bound $(n-2)A_m(d)$ from \cite[Theorem 3.31]{DeB_Voronoi}.  For $n=3$ and $n=4$ the two results offer equal bounds.\end{proposition}

\begin{proof} We first touch on the case $n=3$.  In this special case the lower bound of \cite[Thrm.~3.31]{DeB_Voronoi} is the area of an equilateral triangle with all side lengths $d$; ie.~$D_0(d,d,d)$ in the language of \cite[L.~2.1]{DeB_cyclic_geom}.  The only tree in $\calt_3$ is $T = \{v_T\}$, and chasing the definition of $B_T$ in Proposition \ref{centered lower} gives $B_T(d,d,d) = D_0(d,d,d)$.  We therefore assume below that $n\geq 4$.

For the given $\bfb$ we note that $\sigma(\bfb) = \bfb$ for all $\sigma\in S_n$, so in the bound of Theorem \ref{upgrade} the minimum need only be taken over $T\in\calt_n$.  By \ref{numbers} above, each $T\in\calt_n$ has $n-2$ vertices.  Thus applying the definition of $B_T$, the result will follow by observing for any such $T$ that $D_0(P_v(\bfb_{\cale},\bfb)) \geq A_m(d)$ for each $v\in T^{(0)}-\{v_T\}$; that $D_0(m_{e_0},b_{e_1},b_{e_2}) \geq A_m(d)$, where $e_0$, $e_1$ and $e_2$ in $\cale\cup\calf$ contain $v_T$ with $b_{e_0}$ maximal among the $b_{e_i}$; and that strict inequality holds for some vertex if $n>4$.

The tuple $\bfb_{\cale}\in(\mathbb{R}^+)^{\cale}$ referenced above is defined in Proposition \ref{centered lower} by $b_e = b_e(\bfb)$ for each $e\in\cale$, where $b_e\co(\mathbb{R}^+)^{\calf}\to \mathbb{R}^+$ is from \cite[Lemma 3.19]{DeB_Voronoi}.  For each $v\in T^{(0)}-\{v_T\}$, taking $e_v$ as the initial edge of the arc joining $v$ to $v_T$, $b_{e_v} = b_0(b_{e_1},b_{e_2})$ by the definition of $b_e$ in \cite[L.~3.19]{DeB_Voronoi}, where $b_0$ is the function from \cite[Prop.~1.12]{DeB_cyclic_geom} and $e_1, e_2\in\cale\cup\calf$ are the other two edges containing $v$.  Thus $D_0(P_v(\bfb_{\cale},\bfb)) = D_0(b_{e_1},b_{e_2},b_0(b_{e_1},b_{e_2}))$.

For $v$ and $e_1$, $e_2$ as above, if $e_i\in\calf$ then $b_{e_i} = d$ by our hypothesis here.  
It follows from Property (1) of \cite[L.~3.19]{DeB_Voronoi} that $b_{e_i} > d$ if $e_i\in\cale$.  Corollary 2.4 and Proposition 1.12 of \cite{DeB_cyclic_geom} together imply that if $x\leq x'$ and $y\leq y'$ then $D_0(x,y,b_0(x,y)) \leq D_0(x',y',b_0(x',y'))$, and that strict inequality holds here if $x<x'$ or $y<y'$.  Therefore $D_0(P_v(\bfb_{\cale},\bfb)) \geq D_0(d,d,b_0(d,d))$ for all $v\in T^{(0)}-\{v_T\}$, and the inequality is strict unless $v$ has valence one in $T$.

Property (1) of \cite[L.~3.19]{DeB_Voronoi} can be used here to show the stronger fact that $b_e \geq b_0(d,d)$ for each $e\in\cale$.  By its definition in Proposition \ref{centered lower} and the monotonicity of $b_0$ (\cite[Prop.~1.12]{DeB_cyclic_geom}) we also have $m_{e_0} \geq b_0(d,d)$ for $e_0$, $e_1$ and $e_2$ containing $v_T$ as above.  Therefore again $D_0(m_{e_0},b_{e_1},b_{e_2}) \geq D_0(d,d,b_0(d,d))$ by \cite[Cor.~2.3]{DeB_cyclic_geom}, with strict inequality if $v_T$ does not have valence one in $T$.

As pointed out in the statement of \cite[Thrm.~3.31]{DeB_Voronoi}, $A_m(d)$ is the area of a semicyclic triangle with two sides of length $d$.  But $b_0\co(\mathbb{R}^+)^2\to\mathbb{R}^+$ is characterized by the property that $(x,y,b_0(x,y))\in\calBC_3$ has unique largest entry $b_0(x,y)$ for all $x$ and $y$, where $\calBC_3$ is the space parametrizing semicyclic triangles, so $A_m(d) = D_0(d,d,b_0(d,d))$.  It therefore follows from above that $\min\{B_T(\sigma(\bfb))\}\geq (n-2)A_m(d)$, and that strict inequality holds unless each vertex of $T$ has valence one in $T$.  This in turn only holds if $T$ has two vertices and one edge; ie.~if $n=4$.

It is now straightforward to verify in the case $n=4$ for the lone rooted tree $T$ with one edge, and $\bfb = (d,d,d,d)$, that $B_T(\bfb) = 2 D_0(d,d,b_0(d,d)) = 2A_m(d)$.
\end{proof}

\begin{example}\label{all the same}  Here we consider the $n$-tuple $\bfb_n$ with all entries $b = 2\sinh^{-1}(1)$ (ie.~with $\sinh(b/2) = 1$) and allow $n$ to vary between $4$ and $9$.  The outputs of Theorem 3.31 of \cite{DeB_Voronoi} and Theorem \ref{upgrade} here are recorded  in the table below, in each case truncated after three decimal places.  Recall from Important Note \ref{heythere} that the input for \textit{minimizer.minimize()} corresponding to $\bfb$ is $[1,\hdots,1]$.

The ``tree number'' line below refers to the number in Figure \ref{rooted trees} (or \textit{forest.txt}) of the tree realizing $\min\{B_T(\bfb)\}$.  

\begin{table}[ht]
\begin{tabular}{r | c c c c c c}
  $n$ & 4 & 5 & 6 & 7 & 8 & 9 \\
  \hline
  $(n-2)A_m(b)$ & $2.094$ & $3.141$ & $4.188$ & $5.235$ & $6.283$ & $7.330$ \\
  $\min\{B_T(\bfb_n)\}$ & $2.094$ & $3.295$ & $4.526$ & $5.818$ & $7.107$ & $8.441$ \\
  tree number & 1 & 2 & 4 & 6 & 14 & 20 \\
  \multicolumn{7}{c}{}
\end{tabular}
\caption{The outputs of \cite[Thrm.~3.31]{DeB_Voronoi} and Theorem \ref{upgrade} on $\bfb_n$ as above.}
\label{alla the same}
\end{table}

It is not hard to show using Corollary \ref{semicyclic tri} that for our chosen $b$, $A_m(b) = \pi/3$.  So the exact value of $(n-2)A_m(b)$ is $2\pi/3$ for $n=4$, $\pi$ for $n=5$, and so on.  Note that the difference between bounds grows monotonically with $n$.

A takeaway from the proof of Proposition \ref{bigger than} is that for a particular $T$, with $\bfb=(1,\hdots,1)$, vertices with valence one in $T$ contribute the least to $B_T(\bfb)$.  It is therefore not surprising that for each $n$ the rooted tree realizing $\min\{B_T(\bfb_n)\}$ has the maximum possible number of vertices of valence one among all trees in $\calt_n$.  Note moreover that the root vertex of each such example has maximal valence.\end{example}

\begin{example}\label{one big}  Now we explicitly analyze the simplest case with unequal entries, taking $\bc = (b,b,b,x)$ with $\sinh(b/2) = 1$ as in the previous example.  There is only one tree $T\in\calt_4$, and following the scheme of \textit{treecrawler} we label its frontier as $\{f_0,f_1,f_2,f_3\}$ so that $f_0$ and $f_1$ terminate at the root vertex $v_T$ and $f_2$ and $f_3$ at the other one.  Since $\bc$ has many identical entries, by Remark \ref{coset} there are only two permutations of $\bc$ to check: the identity,  which assigns $x$ to $f_3$ and $b$ to all others, and a permutation $\sigma$ that assigns $x$ to $f_0$.  These are pictured in Figure \ref{one big fig}.

\begin{figure}
\begin{tikzpicture}

\node [left, rectangle, draw=black] at (-1.5,0) {Identity case};
\draw [thick] (0,0) -- (2,0);
\draw [thick] (1,-0.1) -- (1.1,0) -- (1,0.1);
\draw [dashed] (0,0) -- (-0.5,.85);
\draw [dashed] (0,0) -- (-0.5,-.85);
\draw [dashed] (2,0) -- (2.5,.85);
\draw [dashed] (2,0) -- (2.5,-.85);
\draw [fill] (0,0) circle [radius=0.1];
\draw [fill] (2,0) circle [radius=0.1];
\node [left] at (-0.5,.85) {$1$};
\node [left] at (-0.5,-.85) {$X$};
\node [right] at (2.5,.85) {$1$};
\node [right] at (2.5,-.85) {$1$};
\node [above] at (1,0.1) {$\sqrt{X^2+1}$};

\node [right, rectangle, draw=black] at (8.5,0) {Case $\sigma$};
\draw [thick] (5,0) -- (7,0);
\draw [thick] (6,-0.1) -- (6.1,0) -- (6,0.1);
\draw [dashed] (5,0) -- (4.5,.85);
\draw [dashed] (5,0) -- (4.5,-.85);
\draw [dashed] (7,0) -- (7.5,.85);
\draw [dashed] (7,0) -- (7.5,-.85);
\draw [fill] (5,0) circle [radius=0.1];
\draw [fill] (7,0) circle [radius=0.1];
\node [left] at (4.5,.85) {$1$};
\node [left] at (4.5,-.85) {$1$};
\node [right] at (7.5,.85) {$1$};
\node [right] at (7.5,-.85) {$X$};
\node [above] at (6,0.1) {$\sqrt{2}$};

\end{tikzpicture}
\caption{The two cases of Example \ref{one big}.}
\label{one big fig}
\end{figure}
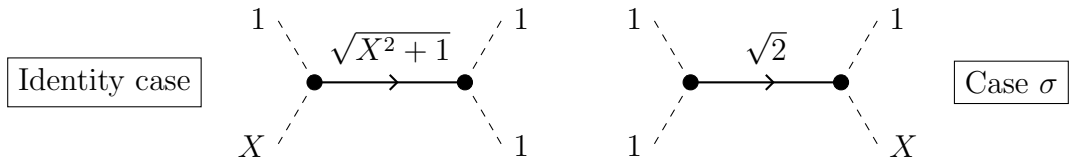

Frontier edges are dashed in the figure, and each such edge $f_i$ is labeled with $\sinh(c_{f_i}/2)$, where $c_f$ is the corresponding entry of $\bc$ (or $\sigma(\bc)$ in the other case).  In particular, $X = \sinh(x/2)$.  The edge $e$ of $T$ points toward $v_T$ (which hence is the right vertex in each case).  It is labeled with $b_e = b_e(\bc)$ as prescribed in Corollary \ref{better upgrade}.

In order to compute $B_T(\bc)$ and $B_T(\sigma(\bc))$ we must first determine the $m_{e_i}$ as in Corollary \ref{better upgrade}, where $e_0$, $e_1$ and $e_2$ are the edges containing $v_T$.  We will take $e_0 = f_0$, $e_1 = f_1$, and $e_2 = e$ in each case.  Then in the identity case, $\sinh(m_{e_2}/2) = \sqrt{X^2+1}$ if $X\leq1$ and otherwise $\sinh(m_{e_2}/2) = \sqrt{2}$, and $m_{e_i} = b$ for $i=0,1$ and any $x$.  This gives:\begin{align*}
  & B_T(\bc) = \left\{\begin{array}{ll}
      2\sin^{-1}\left(\frac{X}{\sqrt{2(X^2+1)}}\right) + 2\cos^{-1}\left(\frac{5+X^2}{4\sqrt{X^2+2}}\right) &
          X \leq 1 \\
      2\sin^{-1}\left(\frac{X}{\sqrt{2(X^2+1)}}\right) + 2\sin^{-1}\left(\frac{1}{2}\right) & X > 1
      \end{array}\right. \end{align*}
In each case above the contribution from $v_T$ is the second summand.  For $X >1$ we simplify by using the observation from the proof of Lemma \ref{semicyclic tri}, that for inputs in $\calBC_3$, Heron's formula reduces to the one from that result.

In case $\sigma$ we have $m_{e_0} = x$, $m_{e_1} = b$ and $\sinh(m_{e_2}/2) = \sqrt{X^2+1}$ if $X\leq 1$; $m_{e_0} = x$, $m_{e_1} = b$ and $\sinh(m_{e_2}/2) = \sqrt{2}$ for $1< X\leq \sqrt{3}$; and $\sinh(m_{e_0}/2) = \sqrt{3}$, $m_{e_1} = b$ and $\sinh(m_{e_2}/2)=\sqrt{2}$ if $X> \sqrt{3}$.  Applying this gives:\begin{align*} & B_T(\sigma(\bc)) = \left\{\begin{array}{ll}
      2\sin^{-1}\left(\frac{1}{2}\right) + 2\sin^{-1}\left(\frac{X}{\sqrt{2(X^2+1)}}\right) & X \leq 1 \\
      2\sin^{-1}\left(\frac{1}{2}\right) + 2\cos^{-1}\left(\frac{5+X^2}{2\sqrt{6(X^2+1)}}\right) & 
          1< X \leq \sqrt{3} \\
      2\sin^{-1}\left(\frac{1}{2}\right) + 2\sin^{-1}\left(\frac{1}{\sqrt{3}}\right) & X > \sqrt{3} \end{array}\right.\end{align*}
Again the contribution from $v_T$ is the second summand.  To understand which of $B_T(\bc)$ or $B_T(\sigma(\bc))$ is larger for various values of $x$, it helps to look at them in a different way.  Below to save space we replace $a$ by $\sinh(a/2)$ for any entry $a$ of an input to the area function $D_0$.\begin{align*}
 & B_T(\sigma(\bc)) - B_T(\bc) = D_0(1,1,\sqrt{2}) - D_0(1,X,\sqrt{2}) > 0 && \mbox{for}\ X\leq 1 \\
 & B_T(\bc) - B_T(\sigma(\bc)) = D_0(1,X,\sqrt{X^2+1}) - D_0(1,X,\sqrt{2}) > 0 && \mbox{for}\ 1<X\leq \sqrt{3} \\
 & B_T(\bd) - B_T(\sigma(\bc)) = D_0(1,X,\sqrt{X^2+1}) - D_0(1,\sqrt{3},\sqrt{2}) > 0 && \mbox{for}\ X>\sqrt{3} \end{align*}
The inequalities in each case follow from Proposition 2.3 of \cite{DeB_cyclic_geom}.  Therefore as long as $X\leq 1$, $\min\{B_T(\bc),B_T(\sigma(\bc))\} = B_T(\bc)$, and otherwise it is $B_T(\sigma(\bc))$.  

Note that the given bound is constant for $x\geq 2\sinh^{-1}(\sqrt{3})$.  Its value, truncated after three decimal places, is $2.278$ for such $x$.  This may be compared with the bound of $2.094$ given for $\bfb_4 = (b,b,b,b)$ in Example \ref{alla the same}.
\end{example}

\begin{remark}\label{reasonable}  The cases of Example \ref{one big} where $X\leq \sqrt{3}$ are distinguished by the fact that taking $\bfb_{\calf} = \bc$ or $\bfb_{\calf} = \sigma(\bc)$ as appropriate to minimize $B_T$, the tuple $\bfb_{\cale}$ produced by Proposition \ref{centered lower} has $(\bfb_{\cale},\bfb_{\calf})\in\overline{\mathit{Ad}}_T(\bfb)$.  Another way of saying this, again using the notation of Proposition \ref{centered lower}, is that $m_{e_0} = b_{e_0}$.

We say that an arbitrary $n$-tuple $\bfb$ with the analogous property is ``geometrically reasonable''.  That is, for $T\in\calt_n$ and $\sigma\in S_n$ such that $B_T(\sigma(\bfb))$ realizes the bounds of Theorem \ref{upgrade}, we should have $(\bfb_{\cale},\sigma(\bfb))\in\overline{\mathit{Ad}}_T(\sigma(\bfb))$.  The reason is that in this case there is a metric triangle complex $\mathfrak{T}$ with combinatorics prescribed by $T$ and geometry by $(\bfb_{\cale},\sigma(\bfb))$, and a map $\mathfrak{T}\to\mathbb{H}^2$ whose image we expect in many cases to be a centered dual cell with edge length collection $\sigma(\bfb)$ and area equalling the bound.

The idea here is that $\mathfrak{T}$ has a triangular face for each $v\in T^{(0)}$ which is a hyperbolic triangle with edge length collection $P_v(\bfb_{\cale},\sigma(\bfb))$, and for each edge $e$ of $T$ the faces of $\mathfrak{T}$ corresponding to its endpoints are glued along their sides with length $b_e$.  A map $\mathfrak{T}\to\mathbb{H}^2$ is determined by choosing an isometric embedding of $P_{v_T}(\bfb_{\cale},\sigma(\bfb))$ and analytically continuing outward from $v_T$ in $T$, forcing the restriction to each $P_v(\bfb_{\cale},\sigma(\bfb))$ to be an isometry.  This map is therefore a local isometry on the interior of $\mathfrak{T}$.

Such a map may fail to be an isometry if there is branching at the vertices, or if different arms of $T$ determine regions of $\mathfrak{T}$ with overlapping images.  Even in the absence of these, a vertex of some $P_v$ may end up inside the circumcircle of some other $P_w$.  But for many geometrically reasonable $\bfb$ these pathologies will not occur, and for these the image of $\mathfrak{T}$ will be a centered dual cell of the set of images $\cals$ of its vertices.\end{remark}

\bibliographystyle{plain}
\bibliography{centered_bound}

\begin{thebibliography}{10}

\bibitem{Beardon}
Alan~F. Beardon.
\newblock {\em The geometry of discrete groups}, volume~91 of {\em Graduate
  Texts in Mathematics}.
\newblock Springer-Verlag, New York, 1983.

\bibitem{Bilinski}
Stanko Bilinski.
\newblock Zur {B}egr\"undung der elementaren {I}nhaltslehre in der
  hyperbolishchen {E}bene.
\newblock {\em Math. Ann.}, 180:256--268, 1969.

\bibitem{DeB_Delaunay}
Jason DeBlois.
\newblock The {D}elaunay tessellation in hyperbolic space.
\newblock Preprint. arXiv:1308.4899.

\bibitem{DeB_Voronoi}
Jason DeBlois.
\newblock The centered dual and the maximal injectivity radius of hyperbolic
  surfaces.
\newblock {\em Geom. Topol.}, 19(2):953--1014, 2015.

\bibitem{DeB_cyclic_geom}
Jason DeBlois.
\newblock The geometry of cyclic hyperbolic polygons.
\newblock {\em Rocky Mountain J. Math.}, 46(3):801--862, 2016.

\bibitem{Even}
Wikipedia entry: The {S}teinhaus--{J}ohnson--{T}rotter algorithm.
\newblock \\{\tt en.wikipedia.org/wiki/SteinhausÐJohnsonÐTrotter{\_}algorithm}.

\bibitem{Knuth}
Donald~E. Knuth.
\newblock {\em The art of computer programming. {V}ol. 1}.
\newblock Addison-Wesley, Reading, MA, 1997.
\newblock Fundamental algorithms, Third edition.

\bibitem{KM}
Sadayoshi Kojima and Yosuke Miyamoto.
\newblock The smallest hyperbolic {$3$}-manifolds with totally geodesic
  boundary.
\newblock {\em J. Differential Geom.}, 34(1):175--192, 1991.

\bibitem{NaatPen}
M.~N{\"a}{\"a}t{\"a}nen and R.~C. Penner.
\newblock The convex hull construction for compact surfaces and the {D}irichlet
  polygon.
\newblock {\em Bull. London Math. Soc.}, 23(6):568--574, 1991.

\bibitem{Stothers}
Wilson {S}tothers' {G}eometry {P}ages.
\newblock {\tt www.maths.gla.ac.uk/wws/cabripages/cabri0.html}.
\newblock In particular, see {\tt
  http://www.maths.gla.ac.uk/wws/cabripages/hyperbolic/overview.html}.

\end{thebibliography}

\end{document}